\documentclass{aims} 
\usepackage{amsmath}
\usepackage{paralist}
\usepackage[misc]{ifsym}
\usepackage{epsfig} 
\usepackage{epstopdf} 
\usepackage[colorlinks=true]{hyperref}
\hypersetup{urlcolor=blue, citecolor=red}
\allowdisplaybreaks

\def\currentvolume{X}
 \def\currentissue{X}
  \def\currentyear{200X}
   \def\currentmonth{XX}
    \def\ppages{X-XX}
     \def\DOI{10.3934/xx.xxxxxxx}

\newtheorem{theorem}{Theorem}[section]
\newtheorem{corollary}[theorem]{Corollary}

\newtheorem{lemma}[theorem]{Lemma}
\newtheorem{proposition}[theorem]{Proposition}

\newtheorem*{problem}{Problem}
\theoremstyle{definition}
\newtheorem{definition}[theorem]{Definition}
\newtheorem{remark}[theorem]{Remark}
\newtheorem{assumption}[theorem]{Assumption}

\newcommand{\ep}{\varepsilon}
\newcommand{\eps}[1]{{#1}_{\varepsilon}}

\newcommand{\Lmax}{\overline{L}}

\newcommand{\R}{\mathbb{R}}

\newcommand*{\V}[1]{\mathbf{#1}} 

\def\vect#1{\mbox{\boldmath$ #1$}}                   

\newcommand{\be}{{\vect e}}
\newcommand{\bj}{{\vect j}}
\newcommand{\bx}{{\vect x}}

\def\vects#1{\mbox{\scriptsize \boldmath$ #1$}}

\newcommand{\ibe}{{\vects e}}
\newcommand{\ibj}{{\vects j}}
\newcommand{\ibone}{{\vects 1}}

\newcommand{\oa}{{\overline{a}}}
\newcommand{\ob}{{\overline{b}}}
\newcommand{\oq}{{\overline{q}}}

\newcommand{\calE}{{\mathcal{E}}}
\newcommand{\calJ}{{\mathcal{J}}}
\newcommand{\mE}{\mathcal{E}} 
\newcommand{\mO}{\mathcal{O}} 

\newcommand{\dx}{\Delta x} 
\newcommand{\dt}{\Delta t} 

\usepackage{float}
\usepackage{subfig}

\usepackage{comment}

\title[Multi-d Hyperbolic Boundary Feedback Control Problem]
{
Numerical Boundary Control of Multi-Dimensional Hyperbolic Equations
} 

\author[Michael Herty, Kai Hinzmann, Siegfried M\"uller and Ferdinand Thein]{}

\subjclass{Primary: 35L50, 35Q93; Secondary: 93D05.}
\keywords{
Boundary feedback control, multi-dimensional hyperbolic problem, numerical dissipation, stabilization.}

\thanks{$^*$Corresponding author: Kai Hinzmann}

\begin{document}
\maketitle

\centerline{\scshape
Michael Herty$^{{\href{mailto:herty@igpm.rwth-aachen.de}{\textrm{\Letter}}}1}$,
Kai Hinzmann$^{{\href{mailto:hinzmann@ssd.rwth-aachen.de}{\textrm{\Letter}}}*1}$,
Siegfried M\"uller$^{{\href{mailto:mueller@igpm.rwth-aachen.de}{\textrm{\Letter}}}1}$,
and Ferdinand Thein$^{{\href{mailto:thein@uni-mainz.de}{\textrm{\Letter}}}2}$}

\medskip

{\footnotesize
 \centerline{$^1$Institut f\"ur Geometrie und Praktische Mathematik, RWTH Aachen, Templergraben 55, Aachen, D-52056, Germany}
} 

\medskip

{\footnotesize
 \centerline{$^2$Johannes Gutenberg-Universit\"at Mainz, Staudingerweg 9, {55128} Mainz, Germany}
}

\bigskip



\begin{abstract}

Existing theoretical stabilization results for linear, hyperbolic multi--dimensional  problems
are extended to the discretized multi-dimensional problems.
In contrast to existing theoretical and numerical analysis in the spatially one--dimensional case the effect of the numerical dissipation is analyzed and explicitly quantified. Further, using dimensional splitting, the numerical analysis is extended to the multi-dimensional case.
	 The findings are confirmed by numerical simulations for low-order and high-order DG schemes both in the one-dimensional and two-dimensional case.

\end{abstract}


\section{Introduction}
\label{sec:Intro}

Stabilization of spatially one--dimensional systems of hyperbolic balance laws has been a subject of many investigations in the mathematical and engineering community, respectively, see, e.g., the monographs \cite{O1,O2,MR2412038,MR2655971}.  In particular, the boundary control of typically one--dimensional hyperbolic systems, like the isentropic Euler equations or the Saint--Venant equations has been a subject of recent research, see e.g.~\cite{G1,G3,G4,W2,W4,W5,L1,L5} and references therein. The boundary control applied has been of feedback type and has been aimed to stabilize the dynamics at a desired equilibrium. The main analytical tool in the spatially one--dimensional setting 
{is}
the analysis of a family of  weighted Lyapunov functions. They have been proposed as  (exponentially) weighted $L^2$-- (or $H^s$--) norms of the solution. Provided that the boundary conditions are dissipative, exponential decay of the Lyapunov function 
{can be established, cf.~\cite{L1,L2,L4,L5}.} 
Furthermore, a comparison to other stability concepts has been provided \cite{L7}.
Extensions towards input--to--state stability have also been investigated, e.g.~in \cite{MR2899713}, as well as extensions towards nonlocal systems, see e.g.~\cite{MR4172728,MR3097269,MR2852203}. While most of the results aim to stabilize the linearized dynamics, there are also results on stability of nonlinear systems, see e.g.~\cite{L5,O1}.
\par
Most of the existing literature  focuses on analytical results and on the spatially one--dimensional case.  Decay rates of the corresponding {\em  numerical discretization } in the spatially one--dimensional case have been established only recently for linear hyperbolic systems with symmetric source term \cite{D1,D2} in $L^2$, for linear hyperbolic balance laws in $H^s$
 \cite{MR3956429,MR4619935},  and for scalar hyperbolic equations using second--order schemes \cite{MR4381729,MR4534579}. Those results can be seen in part as a numerical counterpart to the analytical results reviewed above.
\par
Furthermore, only recently, the Lyapunov stability for {\em multi--dimensional} hyperbolic has been investigated from an analytical perspective using similar techniques as cited above, see \cite{herty2023bound,MR4770847}. In \cite{herty2023bound}, the authors study (general) linear, hyperbolic multi--dimensional systems of balance laws in a bounded domain. Using again a weighted norm as Lyapunov  function, they derive boundary feedback control that stabilizes the dynamics at a steady--state. Exponential decay of the Lyapunov function has been established, however, without explicit rates of the decay. In the multi-dimensional case, there exists also an approach tailored to the shallow-water equations, where boundary conditions have been designed to control the decay of the energy, see \cite{Dia2013}. 
Hyperbolic multi--dimensional systems satisfying the strong stability condition are studied in \cite{MR4759437}. A study on the relation between the approaches given in  \cite{herty2023bound} and \cite{MR4759437} is provided in \cite{herty2024comparison}.
Further, in \cite{Serre2022} a $L^2$-Lyapunov function for multi--dimensional scalar conservation laws has been studied.
\par
The goal of the present paper is to complement the existing stabilization results for linear, hyperbolic multi--dimensional systems ~\cite{MR4770847,herty2023bound,MR4759437} by an analysis of the corresponding discretized systems. Here, we will focus on the exponential stabilization using the family of Lyapunov functions 
introduced in \cite{MR4770847}. In contrast
to the analysis in the spatially one--dimensional case \cite{D1,D2,MR3956429}, we will also investigate and quantify the effect of the numerical dissipation. This dissipation has already been observed in the numerical results, e.g.~in \cite{D1}, but has not been quantified. On coarse grids, this leads to an additional decay of the Lyapunov function compared with the analytical result.  Furthermore, we use dimensional splitting as numerical method. A careful analysis of the splitting is required in order to obtain the exponential decay rate. Following the analytical result of 
\cite{MR4770847},
the numerical analysis will be presented for linear, multi--dimensional systems, and first--order schemes. Numerical results will also be presented using higher--order schemes, in particular, to illustrate  further the effect of the numerical dissipation.

\section{Description of the Problem}
\label{sec:Problem}

We first introduce the continuous setting before turning to the discretization.
To this end we recall the problem and theoretical results of \cite{MR4770847}.

The feedback control problem in the multi--dimensional case is given by a linear system of hyperbolic partial differential equations.
Let $\Omega\subset\R^d$ be a bounded domain with 
piecewise
smooth boundary $\partial\Omega$.
We consider the following IBVP for a linear system of hyperbolic partial differential equations as in \cite{MR4770847}.

\begin{problem}[General IBVP]
\label{prob:IBVP}
For a given bounded domain $\Omega\subset\R^d$   with 
piecewise
smooth boundary $\partial\Omega$
find a solution $\V{w}:[0,T)\times\Omega\to\R^m$
to the IBVP
\begin{subequations}
\label{eq:IBVP}
    \begin{align}
    \label{eq:IBVP-PDE}
	&\partial_t \V{w}(t, \V{x}) + \sum_{k=1}^d \V{A}^{(k)}(\V{x}) \partial_{x_k} \V{w}(t, \V{x}) + \V{B}(\V{x}) \V{w}(t, \V{x}) = \V{0} ,\quad (t, \V{x})\in [0, T)\times\Omega \\
    \label{eq:IBVP-IC}
	&\V{w}(0, \V{x}) = \V{w}_0(\V{x}), \quad \V{x}\in\Omega \\
    \label{eq:IBVP-BC}
	&{w_i}(t, \V{x}) = {u_i}(t, \V{x}) ,\quad 
	(t, \V{x})\in [0, T)\times \Gamma^-_{{i}},\ {i=1,\ldots,d}.
	\end{align}
\end{subequations}
Further, we assume that $\V{A}^{(k)}(\V{x})$ for $k\in\{1,\dots,d \}$ and $\V{B}(\V{x})$ are sufficiently smooth and bounded $m\times m$ real matrices. In addition, all matrices $\V{A}^{(k)}(\V{x})$ are assumed to be diagonal matrices.
\end{problem}

	The boundary feedback control is realized through the function $\V{u}$ in the boundary condition \eqref{eq:IBVP-BC}.
	Here, we prescribe for each component $w_i$ with  $i\in\{1,\ldots,m\}$ boundary values given by some function $\V{u}(t,\V{x})$ at the inflow part of the boundary $\partial\Omega$ characterized by
$
		\Gamma_i^- := \{\V{x}\in\partial\Omega : \V{a}_i(\V{x})\cdot\V{n}(\V{x}) < 0\},
$
	where the vector of diagonal entries of the matrices $\V{A}^{(k)}$ is denoted by
	$\V{a}_i := (a_{ii}^{(1)}, \dots, a_{ii}^{(d)})$.
	The outflow part of the boundary $\partial\Omega$ for component $i$ will be denoted by
$
		\Gamma_i^+ := \{\V{x}\in\partial\Omega : \V{a}_i(\V{x})\cdot\V{n}(\V{x}) \geq 0\}.
$
	\begin{remark}
		Since Dirichlet boundary conditions can only be prescribed at $\Gamma_i^-$ but not at $\Gamma_i^+$ we refer to these boundary parts as controllable and uncontrollable part, respectively.
		Analogous to \cite{MR4770847}, we further divide $\Gamma_i^-$ into the disjoint sets $\mathcal{C}_i\subset\Gamma_i^-$ where we set some control $u_i(t,\V{x})$ componentwise and $\mathcal{Z}_i\subset\Gamma_i^-$ where we set $u_i(t,\V{x})=0$ for the numerical simulations in Section \ref{sec:SimRes}.
	\end{remark}

We are interested in the exponential stabilization 
{at zero}
of the solution $\V{w}$ to 
{IBVP}
\eqref{eq:IBVP}
by applying a boundary feedback control $\V{u}$.
As in \cite{MR4770847},
we consider smooth solutions $\V{w}\in C^1(0,T; H^s(\Omega))^m$ for $s\geq 1+\frac{d}{2}$ of
{IBVP}
\eqref{eq:IBVP}.
If the initial condition and the boundary data are sufficiently smooth, there exists such a solution, see \cite{Dafermos2016}. Moreover, we assume that the solution is unique.
Exponential stability is defined as follows.
\begin{definition}[Exponentially stable]
\label{def:exp-stab}
	A solution $\V{w}\in C^1(0,T; H^s(\Omega))^m$ for $s\geq 1+\frac{d}{2}$ of
	IBVP \eqref{eq:IBVP}
	is called exponentially stable in the $L^2$--sense if and only if there exists a constant $C\in\R_{>0}$ such that
\begin{align*}
  \left\lVert \V{w}(t,\cdot) \right\rVert _{L^2(\Omega)} \leq \operatorname*{exp}(-Ct) \left\lVert \V{w}(0,\cdot) \right\rVert _{L^2(\Omega)} \qquad\forall t\in[0,T).
\end{align*}
\end{definition}

In this setting the following theorem has been obtained.
Note that Def.~\ref{def:exp-stab} can be extended to exponential stability in the $H^s$-norm.
%
%

\begin{theorem}[Exponential stability]
\label{Hauptresultat}
Let $\V{w}\in C^1(0,T; H^s(\Omega))^m$ for $s\geq 1+\frac{d}{2}$ be a solution to
IBVP \eqref{eq:IBVP}.
A Lyapunov function is given by
\begin{align}
\label{eq:Lyapunov}
L(t) = \int_{\Omega} \V{w}(t, \V{x})^T \mathcal{E}(\mu(\V{x})) \V{w}(t, \V{x})\: d\V{x},
\end{align}
where
\begin{equation}
\label{eq:E}
\mathcal{E}(\mu(\V{x})) := \operatorname{diag}(\operatorname{exp}(\mu_1(\V{x})), \dots, \operatorname{exp}(\mu_m(\V{x})))
\end{equation}
for functions $\mu_i(\V{x})\in H^s(\Omega)$ for $i\in\{1,\dots,m\}$ that satisfy
\begin{align}
\label{eq:Bed_mu}
\begin{aligned}
\sum_{k=1}^d \left(\mathcal{M}^{(k)}\V{A}^{(k)} + \partial_{x_k} \V{A}^{(k)}\right) + \mathcal{D} &= -\operatorname{diag}\!\left(C_L^{(i)}\right)  \\
\text{or written rowwise}\quad
\V{a}_i\cdot\nabla\mu_i(\V{x}) + \nabla\cdot\V{a}_i + \mathcal{D}_{ii} &= -C_L^{(i)}, \qquad C_L^{(i)} \geq C_L > 0
\end{aligned}
\end{align}
for some $C_L\in\R_{>0}$, where
\begin{align}
\label{eq:Def_M}
\mathcal{M}^{(k)}(\V{x}):= \operatorname{diag}\!\left(\partial_{x_k} \mu_1(\V{x}), \dots, \partial_{x_k} \mu_m(\V{x})\right).
\end{align}
Here, $\mathcal{D}$ is assumed to be a constant matrix which satisfies
\begin{align}
\label{eq:Bed_D}
-\V{v}^T \left(\V{B}^T\mathcal{E}(\mu(\V{x}))+\mathcal{E}(\mu(\V{x}))\V{B}\right) \V{v} \leq \V{v}^T \mathcal{D}\mathcal{E}(\mu(\V{x}))\V{v} \qquad\forall\V{v}\in\R^n.
\end{align}
Then, for every $\V{u}(t,\V{x}) = (u_1(t,\V{x}),\dots,u_m(t,\V{x}))^T$ such that
\begin{align}
\label{eq:Bed_u}
-\sum_{i=1}^m \int_{\Gamma_i^-} u_i(t,\V{x})^2 (\V{a}_i\cdot\V{n}) \operatorname{exp}(\mu_i(\V{x}))\:d\V{x} \leq \sum_{i=1}^m \int_{\Gamma_i^+} w_i(t,\V{x})^2 (\V{a}_i\cdot\V{n}) \operatorname{exp}(\mu_i(\V{x}))\:d\V{x},
\end{align}
the Lyapunov function satisfies
\begin{align}
\label{eq:AbklingrateBeweis}
\frac{d}{dt}L(t) \leq -C_L L(t)
\quad\text{and}\quad L(t) \le \exp(-C_L t) L(0) \qquad\forall t\in [0,T).
\end{align}
\end{theorem}

The proof of Theorem \ref{Hauptresultat} can be found in
\cite{MR4770847}.

\begin{remark}
In addition to the Hamilton--Jacobi equations,	Theorem \ref{Hauptresultat} can also be applied to the diagonal systems of the multi--dimensional discrete--velocity kinetic models which are discussed in \cite{yang2023boundary} since these are also covered by the IBVP 
{\eqref{eq:IBVP}.}
\end{remark}

%
%
%

\subsection{Feedback Control Problem with Viscosity}
\label{subsec:ControlProblemViscous}

Since any finite volume scheme will introduce numerical viscosity,  we first study the viscous approximation to
IBVP \eqref{eq:IBVP}.
To investigate the effect of viscosity on the decay of the Lyapunov function, we consider a viscous control problem adding  diffusion to \eqref{eq:IBVP-PDE} on a multi-dimensional rectangular domain. For simplicity we confine to the scalar case neglecting zero order terms.
\begin{problem}[Scalar IBVP with diffusion]
\label{prob:IBVP-Diff}
For a given bounded domain $\Omega = [\V{\oa},\V{\ob}]\subset\R^d$
find a solution
$w:[0,T)\times\Omega\to\R$
to the IBVP
%
%
%
\begin{subequations}
\label{eq:IBVP-Diff}
    \begin{align}
    \label{eq:IBVP-Diff-PDE}
	&\partial_t w(t, \V{x}) + \sum_{k=1}^d a_k  \partial_{x_k} w(t, \V{x})
= \oq \sum_{k=1}^d  \partial_{x_k,x_k} w(t, \V{x})  ,\quad (t, \V{x})\in [0, T)\times\Omega\\
    \label{eq:IBVP-Diff-IC}
	&w(0, \V{x}) = w_0(\V{x}) ,\quad \V{x}\in\Omega \\
    \label{eq:IBVP-Diff-BC-0}
	&w(t, \V{x}) =
	u(t, \V{x}) ,\quad (t, \V{x})\in [0, T)\times\Gamma^-,\\
    \label{eq:IBVP-Diff-BC-n}
	&\partial_{\V{n}}w(t, \V{x}) = w_{b,\V{n}}(t, \V{x}) ,\quad (t, \V{x})\in [0, T)\times\Gamma^+.
	 \end{align}
\end{subequations}
The advection velocities are assumed to be $\V{a}=(a_1,\ldots,a_d)\in\R^d$, $a_k\ne0$, and we assume a
positive viscosity coefficient $\oq>0$. By $\partial_{\V{n}}$ we denote the derivative in the outer unit
normal direction $\V{n}$ to the boundary $\partial\Omega$.
\end{problem}

Since  \eqref{eq:IBVP-Diff-PDE} is a second order parabolic PDE, we impose two boundary conditions. Here, we choose
Dirichlet conditions and Neumann conditions at the inflow  boundary $\Gamma^-=\{\V{x}\in\partial\Omega \,:\, \V{a}\cdot \V{n}(\V{x}) < 0\}$ and the outflow boundary $\Gamma^+=\{\V{x}\in\partial\Omega \,:\, \V{a}\cdot \V{n}(\V{x}) > 0\}$, respectively.

Now we can proof the following lemma.

\begin{lemma}
\label{la:decay_pde_with_diffusion}
Let
$w\in C^1((0,T); H^s(\Omega))$
for $s \geq 1+\frac{d}{2}$
be a solution to
IBVP \eqref{eq:IBVP-Diff}.
		The Lyapunov function is given by
		\begin{align}
		  \label{eq:1D_cont_Lyapunov}
			L(t) = \int_{\Omega} w(t, \V{x})^2 \mathcal{E}(\mu(\V{x})) \: d\V{x},
		\end{align}
		where we assume that there exists a function $\mu\in H^s(\Omega)$ satisfying
		\begin{align}
		\label{eq:1D_Bed_mu_diffusion}
			\V{a}\cdot \nabla \mu(\V{x}) = -C_L
		\end{align}
		for some value $C_L\in\R_{>0}$.
		The control $u$ of
		boundary condition \eqref{eq:IBVP-Diff-BC-0} satisfies
		\begin{align}
		\label{eq:1D_Bed_u_diffusion}
			\begin{cases}
				u^2(t,\V{x}_L^k) \mathcal{E}(\mu(\V{x}_L^k)) \leq  w^2(t,\V{x}_R^k-0) \mathcal{E}(\mu(\V{x}_R^k))\,, & a_k>0\\
				 u^2(t,\V{x}_R^k) \mathcal{E}(\mu(\V{x}_R^k)) \leq  w^2(t,\V{x}_L^k+0) \mathcal{E}(\mu(\V{x}_L^k))\,, & a_k<0
			\end{cases}
		\end{align}
		at the inflow boundary $\Gamma^-_k$ in the $k$th coordinate direction.
		Here $\V{x}_L^k$ and $\V{x}_R^k=\V{x}_L^k - (\ob_k-\oa_k)\be_k $ denote opposing points at the left and the right boundary surface in the $k$th direction, respectively.\\
		Then the Lyapunov function is bounded by
		\begin{align}
		\label{eq:decay_rate_diffusion}
			L(t) \leq \operatorname*{exp}(-C_L t) L(0) + \int_0^t R(s;\bar{q})\operatorname*{exp}(-C_L (t - s))\:ds
		\end{align}
		where $R(t;\bar{q}) := \int_{\Omega} 2\bar{q} w\mathcal{E}(\mu) \nabla \cdot \nabla w\:d\V{x}$.
	\end{lemma}

	For the proof of Lemma \ref{la:decay_pde_with_diffusion} we proceed analogously to the proof of Theorem \ref{Hauptresultat}. An additional term occurs accounting for the viscosity coefficient $\bar{q}$ that needs to be estimated.

	\begin{proof}
		Inserting \eqref{eq:IBVP-Diff-PDE} in the time derivative of the Lyapunov function \eqref{eq:1D_cont_Lyapunov}  yields
		\begin{align}
		\label{eq:Proof1D_diff_dL}
		  \frac{d}{dt}L(t) =
		  -2 \int_{\Omega} ( \nabla \cdot (\V{a} w) - \oq \nabla\cdot\nabla w) w \mathcal{E}(\mu)\:d\V{x}.
		\end{align}
		By the product rule we have
		\begin{align*}
		   2 w \mathcal{E}(\mu) \nabla\cdot(\V{a} w) =
		   \nabla\cdot (\V{a} w^2 \mathcal{E}(\mu)) - w^2 \V{a} \cdot \nabla \mathcal{E}(\mu).
		\end{align*}
		From this we conclude
		\begin{align}
		\label{eq:Proof1D_diff_dL-2}
			\frac{d}{dt}L(t) =
			-\int_{\Omega} \nabla \cdot(\V{a} w^2 \mathcal{E}(\mu)) -
			                w^2 \V{a} \cdot \nabla \mathcal{E}(\mu) -
			               2 \oq w \mathcal{E}(\mu) \nabla \cdot \nabla w \:d\V{x}.
		\end{align}
By means of the Gaussian theorem we obtain for the divergence part
		\begin{align*}
			\int_{\Omega} \nabla \cdot(\V{a} w^2 \mathcal{E}(\mu))\:d\V{x} =
			\sum_{k=1}^d a_k \left( \int_{\Gamma^k}(w^2 \mathcal{E}(\mu))(t,\V{x}_R^k(\V{s}))  - (w^2 \mathcal{E}(\mu))(t,\V{x}_L^k(\V{s})) \:d\V{s} \right) ,
		\end{align*}
		where $\V{x}_L^k(\V{s})$ and $\V{x}_R^k(\V{s})$ denote the boundary points of the surfaces
		$x_k=\oa_k$ and $x_k=\ob_k$ in the $k$th coordinate direction, respectively,   parameterized by means of the $(d-1)$--dimensional domain $\Gamma^k$.
		From the boundary condition \eqref{eq:IBVP-Diff-BC-0} and the assumptions \eqref{eq:1D_Bed_u_diffusion} we conclude
		\begin{align*}
			&\int_{\Omega} \nabla \cdot(\V{a} w^2 \mathcal{E}(\mu))\:d\V{x} = \\
			&\sum_{k=1}^d a_k
			\begin{cases}
			  \int_{\Gamma^k} w^2(t, \V{x}_R^k)\mathcal{E}(\mu(\V{x}_R^k)) - u^2(t, \V{x}_L^k)\mathcal{E}(\mu(\V{x}_L^k)) \:d\V{s}\,,& a_k>0\\
			  \int_{\Gamma^k} u^2(t, \V{x}_R^k)\mathcal{E}(\mu(\V{x}_R^k)) - w^2(t,\V{x})_L^k)\mathcal{E}(\mu(\V{x}_L^k))\:d\V{s}\,, & a_k<0
			\end{cases}
			\ge 0.
		\end{align*}
	For the remaining integral in \eqref{eq:Proof1D_diff_dL-2} we obtain
		\begin{align*}
                   I(t) :=
                   \int_{\Omega} w^2 \V{a} \cdot \nabla \mathcal{E}(\mu) +
			               2 \oq w \mathcal{E}(\mu) \nabla \cdot \nabla w \:d\V{x} =
                   \int_{\Omega} -C_L  w^2  \mathcal{E}(\mu)  +
			               2 \oq w \mathcal{E}(\mu) \nabla \cdot \nabla w \:d\V{x} .
		\end{align*}
		By using ${\nabla} \mathcal{E}(\mu) = \mathcal{E}(\mu) {\nabla} \mu$
	as well as \eqref{eq:1D_Bed_mu_diffusion}
		we may estimate \eqref{eq:Proof1D_diff_dL-2} by
		\begin{align*}
			\frac{d}{dt}L(t) \leq I(t) = -C_L L(t) + R(t; \bar{q}).
		\end{align*}
		Applying Gronwall's Lemma gives
		\begin{align*}
			L(t) \leq \operatorname*{exp}(-C_L t) L(0) + \int_0^t R(s;\bar{q})\operatorname*{exp}(-C_L (t - s))\:ds
		\end{align*}
	\end{proof}

	\begin{remark}
		Note that $\bar{q}=0$ implies $R(t;0) = 0$, i.e., in this case we obtain
		the same decay as in Theorem \ref{Hauptresultat}
		for the inviscid feedback control problem.
		Furthermore, we note that diffusion has no  impact on the decay rate $C_L$, but the decay will be perturbed  by a residual.
		Later, we will also quantify the term $R(\cdot\ ;\bar{q})$ numerically.
		Therefore, we expect for a numerical scheme a 
                {faster}
		decay of the Lyapunov function compared with the continuous result.
	\end{remark}

\subsection{Numerical analysis of the discretized problem}
\label{subsec:Analysis}

In Sect.~\ref{subsec:ControlProblemViscous} we showed that viscosity in the feedback control problem, see Eq.~\eqref{eq:IBVP-Diff}, does not affect the exponential decay rate but influences the overall decay behavior due to an additional residual term, see Eq.~\eqref{eq:decay_rate_diffusion}. We now investigate whether numerical dissipation in the inviscid control problem \eqref{eq:IBVP} has a similar effect on the numerical decay behavior. For this purpose, we consider a simplified configuration of a multi--dimensional scalar feedback control problem.
This is similar to IBVP \eqref{eq:IBVP-Diff} where again we neglect source terms in \eqref{eq:IBVP-PDE}. Note that neglecting the source terms the IBVP \eqref{eq:IBVP} decouples in $m$ independent scalar problems justifying to consider only a single scalar problem even in the case of higher spatial dimensions.

\begin{problem}[Scalar IBVP]
\label{prob:IBVP-scalar}
For a given bounded domain $\Omega = [\V{\oa},\V{\ob}]\subset\R^d$
find a solution
$w:[0,T)\times\Omega\to\R$
to the IBVP
\begin{subequations}
\label{eq:IBVP-scalar}
    \begin{align}
    \label{eq:IBVP-scalar-PDE}
	\partial_t w(t, \V{x}) + \sum_{k=1}^d a_k  \partial_{x_k} w(t, \V{x})
= 0 &,\quad (t, \V{x})\in [0, T)\times\Omega\\
    \label{eq:IBVP-scalar-IC}
	w(0, \V{x}) = w_0(\V{x}) &,\quad \V{x}\in\Omega \\
    \label{eq:IBVP-scalar-BC}
	w(t, \V{x}) = u(t, \V{x}) &,\quad (t, \V{x})\in [0, T)\times\Gamma^-.
	 \end{align}
\end{subequations}
Here, the advection velocities are assumed to be $\V{a}=(a_1,\ldots,a_d)\in\R^d$, $a_k\ne0$.  Note that boundary conditions can only be imposed on the inflow boundary $\Gamma^-=\{\V{x}\in\partial\Omega \,:\, \V{a}\cdot \V{n}(\V{x}) < 0\}$ where $\V{n}$ denotes the unit outer normal direction $\V{n}$ to the boundary $\partial\Omega$.
\end{problem}

To approximate \eqref{eq:IBVP-scalar} numerically, we perform in each time step a dimensional splitting where for each direction we solve a quasi-one dimensional problem. For this purpose, we discretize the intervals
$[\oa_k,\ob_k]$, $k=1,\ldots,d$,  for the $k$th direction by
\begin{align*}
  & x_{j-1/2}^k  = \oa_k +(j-1) \Delta x_k,\
  x_j^k = (x_{j-1/2}^k + x_{j+1/2}^k)/2,\
  \Delta x_k = (\ob_k - \oa_k)/M_k,\\
  & I_j^k = (x_{j-1/2}^k, x_{j+1/2}^k),\ j\in\{0,\ldots,M_k+1\} .
\end{align*}
From this we determine the multi--dimensional discretization of the domain $\Omega$
\begin{align*}
 \bx_{\ibj}  = (x_{j_1}^1,\ldots,x_{j_d}^d),\
   V_{\ibj} = (\bx_{\ibj-\ibone/2},\bx_{\ibj+\ibone/2})
          = \bigotimes_{k=1}^d I_j^k, \
          \bj\in \calJ^+ = \bigotimes_{k=1}^d \{0,\ldots,M_k+1 \}
\end{align*}
Note that  $\Omega = \bigcup_{\ibj\in \calJ} V_{\ibj} $ with $\calJ = \bigotimes_{k=1}^d \{1,\ldots,M_k \}$. The elements corresponding to the set $\calJ^+\backslash \calJ$ denote the ghost cells at the boundary.
In the $k$th direction the ghost cells are
\begin{align*}
  &\calE^k_L = \{ (l_1,\ldots,l_{i-1},0,l_{i+1},\ldots,l_d) \,:\, l_j\in\{1,\ldots,M_j\},\ j\ne k\},\\
  &\calE^k_R = \{ (l_1,\ldots,l_{i-1},M_i+1,l_{i+1},\ldots,l_d) \,:\, l_j\in\{1,\ldots,M_j\},\ j\ne k \} .
\end{align*}
The bounded time interval $[0,T]$ is discretized by $t_n = n\Delta t$, $n\in\{0,\dots,N\}$, with
$N\Delta t = T$. The time step $\Delta t$ is chosen such that the $\operatorname*{CFL}$ condition
\begin{align}
\label{eq:CFL-cond}
 \max_{k=1,\ldots,d} \lambda_k\left\lvert a_k \right\rvert =
 \max_{k=1,\ldots,d} \frac{\Delta t}{\Delta x_k} \left\lvert a_k \right\rvert \leq 1
 \qquad\text{for}\qquad\lambda_k =
 \frac{\Delta t}{\Delta x_k}
\end{align}
holds.

In  each direction $k=1,\ldots,d$ we successively apply a linear 3--point finite volume scheme
\begin{align}
\label{eq:FVS-split}
   & w^{n,k}_{\ibj} =
      w^{n,k-1}_{\ibj} - \frac{a_k \lambda_k}{2} ( w^{n,k-1}_{\ibj +\ibe_k} - w^{n,k-1}_{\ibj -\ibe_k} ) +  \frac{q_k}{2}(w^{n,k-1}_{\ibj -\ibe_k} - 2 w^{n,k-1}_{\ibj} + w^{n,k-1}_{\ibj +\ibe_k} )
\end{align}
with boundary conditions at the left and right boundary in the $k$th direction
\begin{subequations}
\label{eq:FVS-split-bc}
\begin{align}
\label{eq:FVS-split-bc-left}
   & w^{n,k-1}_\ibj =
     \begin{cases}
u(t_n,x_{\ibj+\ibe_k/2})\,, & a_k >0 \\
       w^{n,k-1}_{\ibj+\ibe_k}\,, & a_k < 0
     \end{cases},\ \bj \in \calE_L^k, \\
\label{eq:FVS-split-bc-right}
   & w^{n,k-1}_\ibj =
     \begin{cases}
u(t_n,x_{\ibj-\ibe_k/2})\,, & a_k <0 \\
        w^{n,k-1}_{\ibj-\ibe_k}\,, & a_k > 0
     \end{cases},\ \bj \in \calE_R^k.
\end{align}
\end{subequations}

Note that there is only  flow in the $k$th direction but no flow in the other directions. Thus, the update along all the lines in the $k$th direction are independent of each other. 
Essentially all linear 3--point finite volume schemes in 1D can be written in the form 
\eqref{eq:FVS-split} only differing in the choice of the numerical viscosity coefficient $q_k$.
In particular, if the lines are unbounded  the scheme
is $l_2$--stable if the numerical viscosity coefficient  is chosen as 
\begin{align*}
(\lambda_k a_k)^2 \le q_k \le 1.
\end{align*}
 For $q_k=1$ it corresponds to the Lax--Friedrichs scheme. In general it is of first order except for $q_k=(\lambda_k a_k)^2$ where it is of  2nd order corresponding to the Lax-Wendroff scheme, see \cite{Godlewski-Raviart:91}.

We emphasize that the $k^{th}$ intermediate step does not correspond to the intermediate time $t_{n+(k-1)/d}$ because in each step we perform a full time step $\Delta t$ instead of $\Delta t/d$.
This becomes important in equation \eqref{eq:discr_Bed_mu-multid}.
Here, $w^{n,0}_{\ibj} = w^{n}_{\ibj}$ and $w^{n,d}_{\ibj} = w^{n+1}_{\ibj}$ correspond to $t_n$
and $t_{n+1}$, respectively.
	The initial condition $w_0$ is approximated in each cell by the average over this cell, i.e.
\begin{equation}
\label{eq:FVS-split-ic}
w_\ibj^0 = \frac{1}{|V_\ibj|} \int_{V_{\ibj}} w_0(\bx)\: d\bx \quad\text{for}\quad \ibj\in\calJ.
	\end{equation}

The Lyapunov function \eqref{eq:Lyapunov} is discretized at time $t_n$  and 
{at} 
the
intermediate stages of the dimensionally splitted scheme
by means of the midpoint rule, i.e.
\begin{align}
\label{eq:disc_Lyapunov-multid}
\mathcal{L}^n := \sum_{\ibj\in\calJ} \left(w_\ibj^n\right)^2\mathcal{E}_\ibj |V_\ibj|,\quad
\mathcal{L}^{n,k} := \sum_{\ibj\in\calJ} \left(w_\ibj^{n,k}\right)^2\mathcal{E}_\ibj |V_\ibj|,\
k=1,\ldots,d,
\end{align}
where $\mathcal{E}_\ibj := \operatorname*{exp}(\mu(\bx_\ibj))$ and $|V_\ibj|=\prod_{k=1}^d \Delta x_k = |V|$.
Note that $\mathcal{L}^n = \mathcal{L}^{n,0}$  and $\mathcal{L}^{n+1} = \mathcal{L}^{n,d}$.

For the discrete Lyapunov function we can verify the following result.

	\begin{theorem}[Decay of the discrete Lyapunov function in multi--dimension]
		\label{thm:decay_discr_Lyapunov_theorem-multid}
		Let $\left(w_\ibj^n\right)_{\ibj\in\calJ^+, n=0,\dots,N}$ be a bounded approximate solution to
		IBVP \eqref{eq:IBVP-scalar}
		computed by the numerical scheme \eqref{eq:FVS-split} with boundary data \eqref{eq:FVS-split-bc} and initial data \eqref{eq:FVS-split-ic}
		for some $\Delta t>0$ which satisfies
		\begin{align}
		\label{eq:Bed_delta_t-multid}
		\max_{k=1,\ldots,d} \frac{\Delta t}{\Delta x_k} \left\lvert a_k \right\rvert \leq 1
			 \quad\text{and}\quad
	        \left(1-C_L\Delta t\right) > 0.
		\end{align}
		The discrete Lyapunov function is given by \eqref{eq:disc_Lyapunov-multid}
		with
		\begin{align}
		\label{eq:E-mu}
		   \mathcal{E}_\ibj := \operatorname*{exp}(\mu(\bx_\ibj)),\quad
		   \mu(\bx)=\sum_{k=1}^d \mu_k(x_k)
		\end{align}
		where we assume that there exist  functions $\mu_k\in H^s(\Omega)$, $s\ge 1  + \frac{d}{2}$,
		 satisfying
		\begin{align}
		  \label{eq:discr_Bed_mu-multid}
			a_k\mu_k'(x) = -\frac{C_L}{d},\quad k=1,\ldots, K
		\end{align}
		for some constant value $C_L>0$.
		The boundary values \eqref{eq:FVS-split-bc} corresponding to the boundary conditions \eqref{eq:IBVP-scalar-BC} are for all directions $k=1,\ldots,d$ and all time steps $n=0,\ldots,N$
\begin{align}
\label{eq:ed_BC_discr-multid}
    w^{n,k-1}_\ibj =
     \begin{cases}
        u^{n,k}_\ibj\,, & a_k >0 \\
        w^{n,k-1}_{\ibj+\ibe_k}\,, & a_k < 0
     \end{cases},\ \bj \in \calE_L^k, \quad
    w^{n,k-1}_\ibj =
     \begin{cases}
        u^{n,k}_\ibj\,, & a_k <0 \\
        w^{n,k-1}_{\ibj-\ibe_k}\,, & a_k > 0
     \end{cases},\ \bj \in \calE_R^k,
\end{align}
		where $u^{n,k}_\ibj$ satisfies
		\begin{align}
		  \label{eq:Bed_u_discr-multid}
			\begin{cases}
				a_k\left(u^{n,k}_\ibj\right)^2 \frac{\mathcal{E}_{\ibj} + \mathcal{E}_{\ibj+\ibe_k}}{2} \leq a_k\left(w_{\ibj+(M_k+1)\ibe_k}^n\right)^2 \frac{\mathcal{E}_{\ibj+M_k\ibe_k} + \mathcal{E}_{\ibj+(M_k+1)\ibe_k}}{2}\,, & a_k>0,\ \bj \in \calE_L^k\\
				a_k\left(u^{n,k}_\ibj\right)^2\frac{\mathcal{E}_\ibj + \mathcal{E}_{\ibj-\ibe_k}}{2} \geq a_k\left(w_{\ibj-(M_k+1)\ibe_k}^n\right)^2\frac{\mathcal{E}_{\ibj-(M_k+1)\ibe_k} + \mathcal{E}_{\ibj-M_k\ibe_k}}{2} \,,& a_k<0,\ \bj \in \calE_R^k
			\end{cases} .
		\end{align}
		Then the discrete Lyapunov function is bounded by
		\begin{align}
		  \label{eq:discr_decay_rate-multid}
			\mathcal{L}^n \leq
\operatorname*{exp}(-C_L n\Delta t)\mathcal{L}^0 + \Delta t\sum_{i=1}^n \sum_{k=1}^d \operatorname*{exp}(-C_L (k-1+(i-1)d)\Delta t/d) \mathcal{R}^{n-i,d-k}
		\end{align}
		for a residual 
$\mathcal{R}^{n,k} = \mathcal{R}^{n,k}(\Delta t, \Delta x_{k}; a_{k},q_k)$, $k=1,\ldots,d$,
defined in \eqref{residual_R_n_k}.
	\end{theorem}

	Some remarks are in order.

	\textbf{Assumption on the function $\mu$.}
		Note that the condition on 
$\mu(\V{x})$ in \eqref{eq:E-mu} is only needed due to
the dimensional splitting.

\textbf{Comparison of the decay for the discrete and continuous Lyapunov functions}.
Comparing the estimate \eqref{eq:discr_decay_rate-multid} of the discrete Lyapunov function  corresponding to the discretization of the inviscid feedback control problem and
the estimates  \eqref{eq:decay_rate_diffusion} and \eqref{eq:AbklingrateBeweis} for the continuous Lyapunov function for the viscous and the inviscid  feedback control problems, see
IBVP \eqref{eq:IBVP-Diff}
and
\eqref{eq:IBVP}, respectively,
we note that
 (i) the exponential decay rate is the same and
 (ii) the numerical viscosity in the discretization causes a
 residual term in the discrete decay that is an approximation of the continuous residual of the viscous problem.

\textbf{Boundedness of discrete data.}
		The boundedness of the discrete solution $\left(w_i^n\right)_{i=0,\dots,M+1}^{n=0,\dots,N}$ follows from the fact that the  numerical flux is monotone if
the CFL condition \eqref{eq:CFL-cond} holds. This holds provided the initial data and the boundary data given by the Dirichlet boundary condition are bounded. In this case  the data at the new time level $t_{n+1}$ are bounded by the minimum and the maximum of the cell data and the boundary data at the old time level $t_n$, see \cite{Godlewski-Raviart:91}.

\textbf{Consistency of control.}
Note the similarity of the constraints on the control in the discrete condition  \eqref{eq:Bed_u_discr-multid} and the continuous condition  \eqref{eq:1D_Bed_u_diffusion}.
In the discrete case the mean values $(\mathcal{E}_{\ibj} + \mathcal{E}_{\ibj+\ibe_k})/2$ and $(\mathcal{E}_{\ibj-(M_k+1)\ibe_k} + \mathcal{E}_{\ibj-M_k\ibe_k})/2$ approximate the point values $\mE(\mu(\bx_L^k))$ and $\mE(\mu(\bx_R^k))$.
Furthermore, we note that
\eqref{eq:E-mu} and \eqref{eq:discr_Bed_mu-multid} imply
\eqref{eq:1D_Bed_mu_diffusion}.

\textbf{Consistency of decay rate.}
The estimate \eqref{eq:discr_decay_rate-multid} reveals that numerical viscosity has no influence on the decay rate of the discrete Lyapunov function \eqref{eq:disc_Lyapunov-multid}. This is consistent to the decay rate of the  exact Lyapunov \eqref{eq:Lyapunov} in the estimate \eqref{eq:AbklingrateBeweis}.
Instead, the numerical viscosity causes a perturbation in the decay of the discrete Lyapunov function that is similar to the perturbation in the decay of the  exact Lyapunov function for the diffusive
IBVP \eqref{eq:IBVP-Diff}. In \cite{Hinzmann:24} consistency of the discrete estimate \eqref{eq:discr_decay_rate-multid}
and the continuous estimates \eqref{eq:AbklingrateBeweis} (no diffusion) and \eqref{eq:decay_rate_diffusion} (with diffusion) is verified in the 1D case. The
discrete perturbation vanishes for $\Delta x\to 0$, $\lambda=const$, i.e., $\oq=0$ in
IBVP \eqref{eq:IBVP-Diff}, whereas the discrete perturbation is consistent with perturbation in the diffusive case for $\Delta x\to$, $\lambda/\Delta x=const$, if  $\oq=1/2 q\lambda/\Delta x$, i.e., the viscosity coefficient $\oq$ is proportional
to the numerical viscosity coefficient $q$.

\textbf{Accuracy of discretization and accuracy of decay.}
As confirmed
by numerical investigations in Section \ref{subsec:MultidHighOrder} a higher order discretization of the PDE does not necessarily lead to a higher order approximation of the exact rate
$L(0)\operatorname*{exp}(-C_L t)$ in Theorem \ref{Hauptresultat}. However, a higher order approximation introduces less numerical viscosity affecting the perturbation term in the estimate of the decay \eqref{eq:disc_Lyapunov-multid} of the discrete Lyapunov function \eqref{eq:disc_Lyapunov-multid}. Typically this results in a decay
that is closer to the exact decay.

\section{Proof of Theorem \ref{thm:decay_discr_Lyapunov_theorem-multid}}
\label{sec:Proof}

We now prove Theorem \ref{thm:decay_discr_Lyapunov_theorem-multid}  in two steps. In Sect.~\ref{sec:Proof-1D} we first consider the one--dimensional case. This  result is extended to the multi--dimensional case in Sect.~\ref{subsec:Analysis-multiD} using a dimensional splitting argument.

\subsection{Setting and one--dimensional analysis}
	\label{sec:Proof-1D}

For ease of representation we state
the problem, its discretization and the theorem for the one--dimensional case.
\begin{problem}[One--dimensional scalar IBVP]
\label{prob:IBVP-scalar-1D}
For a given bounded domain $\Omega = [\oa,\ob]\subset\R$
find a solution
$w:[0,T)\times\Omega\to\R$
to the IBVP
\begin{subequations}
\label{1D_IBVP}
    \begin{align}
\label{1D_IBVP_A}
	&\partial_t w(t, x) + a \partial_{x} w(t, x)
= 0 ,\quad (t, x)\in [0, T)\times\Omega\\
\label{1D_IBVP_B}
	&w(0, x) = w_0(x) ,\quad x\in\Omega \\
    \label{1D_IBVP_C}
	&w(t, x) =
	u(t, x) ,\quad (t, x)\in [0, T)\times\Gamma^-
	 \end{align}
\end{subequations}
with non-vanishing advection velocity  $a\ne 0$ and boundary condition  only imposed on the inflow boundary $\Gamma^-=\{\oa\}$ or $\Gamma^-=\{\ob\}$ if $a>0$ and $a<0$, respectively.
\end{problem}
This problem is approximated on a uniform discretization for the interval $\Omega=[\oa,\ob]$ by the elements
\begin{align*}
  & x_{j-1/2}  = \oa +(j-1) \Delta x,\
  x_j = (x_{j-1/2} + x_{j+1/2})/2,\
  \Delta x = (\ob - \oa)/M,\\
  & I_j = (x_{j-1/2}, x_{j+1/2}),\ j\in \calJ^+:=\{0,\ldots,M+1\} .
\end{align*}
Note that  $\Omega = \bigcup_{j\in \calJ} V_{j} $ with $\calJ = \{1,\ldots,M \}$. The elements corresponding to the set $\calJ^+\backslash \calJ = \{0,M+1\}$ denote the ghost cells at the boundary.
The bounded time interval $[0,T]$ is discretized by $t_n = n\Delta t$, $n\in\{0,\dots,N\}$, with
$N\Delta t = T$. The time step $\Delta t$ is chosen such that the $\operatorname*{CFL}$ condition
\begin{align*}
  \lambda\left\lvert a \right\rvert =
  \frac{\Delta t}{\Delta x} \left\lvert a \right\rvert \leq 1
 \qquad\text{for}\qquad\lambda =
 \frac{\Delta t}{\Delta x}
\end{align*}
holds.

In the 1D case a general
3--point finite volume scheme \eqref{eq:FVS-split} in viscous form, see \cite{Godlewski-Raviart:91}, Chap.~3.3, reads
\begin{align}
\label{3point_scheme}
   & w^{n_+1}_j =
      w^{n}_j - \frac{a \lambda}{2} ( w^{n}_{j +1} - w^{n}_{j -1} ) +  \frac{q}{2}(w^{n}_{j -1} - 2 w^{n}_{j} + w^{n}_{j +1} )
\end{align}
with Courant number $\lambda a$ and numerical viscosity coefficient $q$.
The boundary conditions at the left and right boundary given by \eqref{eq:FVS-split-bc} are determined by
\begin{align}
\label{eq:FVS-split-bc-1D}
    w^{n}_0 =
     \begin{cases}
		u(t_{n},x_{1/2})\,, & a >0 \\
        w^{n}_{1}\,, & a < 0
     \end{cases},\
    w^{n}_{M+1} =
     \begin{cases}
		u(t_{n},x_{M+1/2})\,, & a <0 \\
        w^{n}_{M}\,, & a > 0
     \end{cases}.
\end{align}
	The initial condition $w_0$ is approximated in each cell by the average over this cell, i.e.
\begin{align}
\label{discr_inital_data}
w_j^0 = \frac{1}{|V_j|} \int_{V_{j}} w_0(x)\: dx \quad\text{for}\quad j\in\calJ.
\end{align}

	We control the solution at the inflow boundary and observe at the 
{outflow boundary. We apply}
	a discrete boundary feedback control $u^n = u^n\!\left(w_{M+1}^n\right)$ and $u^n = u^n\!\left(w_{0}^n\right)$ if $a>0$ and $a<0$, respectively, by setting $w_0^n := u^n$ and  $w_{M+1}^n := u^n$ in the left (right) ghost cell and extend the solution constantly at the right (left) ghost cell  by setting $w_{M+1}^n = w_M^n$ and $w_{0}^n = w_1^n$, respectively.
Thus, the control enters the flow field by
	\begin{equation}
	\label{3point_scheme_w1}
		\begin{cases}
			w_1^{n+1} = w_1^n - \frac{\lambda a}{2}\left(w_{2}^n - u^n\right) + \frac{q}{2}\left(w_{2}^n-2w_1^n+u^n\right)\,, & a>0\\
			w_M^{n+1} = w_M^n - \frac{\lambda a}{2}\left(u^n - w_{M-1}^n\right) + \frac{q}{2}\left(u^n-2w_M^n+w_{M-1}^n\right)\,, & a<0
		\end{cases} .
	\end{equation}
	At the outflow boundary we have
	\begin{equation}
	\label{3point_scheme_wM}
		\begin{cases}
			w_M^{n+1} = w_M^n - \frac{\lambda a}{2}\left(w_{M}^n - w_{M-1}^n\right) + \frac{q}{2}\left( w_{M}^n-2w_M^n+w_{M-1}^n\right)\,, & a>0\\
			w_1^{n+1} = w_1^n - \frac{\lambda a}{2}\left(w_{2}^n - w_1^n\right) + \frac{q}{2}\left(w_{2}^n-2w_1^n+w_1^n\right)\,, & a<0
		\end{cases}
	\end{equation}
	and then we set
	\begin{equation}\label{3point_scheme_wMp1}
		\begin{cases}
			w_{M+1}^{n+1} = w_M^{n+1}\,, & a>0\\
			w_0^{n+1} = w_1^{n+1}\,, & a<0
		\end{cases} .
	\end{equation}

	The Lyapunov function \eqref{eq:Lyapunov} is discretized at time $t_n$   by means of the midpoint rule, i.e.
\begin{align}
\label{disc_Lyapunov}
\mathcal{L}^n := \sum_{j\in\calJ} \left(w_j^n\right)^2\mathcal{E}_j \Delta x,
\end{align}
where 
$\mathcal{E}_j := \operatorname*{exp}(\mu(x_j))$.
If $w$ is continuous, then $\mathcal{L}^{n}$ is a second order approximation of  $L(t_n)$, i.e.,
	\begin{align*}
		L(t_n) = \mathcal{L}^n + \mO(\dx^2).
	\end{align*}

\subsubsection{Decay rate of the discrete Lyapunov function}

	With the one--dimensional setting described above
	 the decay of the discrete Lyapunov function can be estimated by the following theorem that is an analogue to Theorem \ref{thm:decay_discr_Lyapunov_theorem-multid}.

	\begin{theorem}
		\label{decay_discr_Lyapunov_theorem}
		Let $\left(w_i^n\right)_{i=0,\dots,M+1}^{n=0,\dots,N}$ be a bounded discrete solution to
		IBVP \eqref{1D_IBVP}
		computed by the numerical scheme \eqref{3point_scheme}-\eqref{3point_scheme_wMp1}
		for some $\Delta t>0$ which satisfies
		\begin{equation}\label{Bed_delta_t}
			\frac{\Delta t}{\Delta x}\left\lvert a \right\rvert \leq 1 \quad\text{and}\quad \left(1-C_L\Delta t\right) > 0.
		\end{equation}
		The discrete Lyapunov function is given by \eqref{disc_Lyapunov}
		with $\mathcal{E}_{j} := \operatorname*{exp}(\mu(x_{j}))$
		where we assume that there exists a function $\mu(x)\in H^s([\oa,\ob])$, $s\ge \frac{3}{2}$.  satisfying
		\begin{equation}\label{discr_Bed_mu}
			a\mu_x(x) = -C_L
		\end{equation}
		for some constant value $C_L>0$.
		The boundary condition for \eqref{1D_IBVP_C} is realized by setting
		\begin{equation}\label{Bed_BC_discr}
			\begin{cases}
				w_0^n =	u^n\,, & a>0\\
				w_{M+1}^n =	u^n\,, & a<0
			\end{cases}
			\quad\text{for } n\in\{0,\dots, N\},
		\end{equation}
		where $u^n$ satisfies
		\begin{equation}
		\label{Bed_u_discr}
			\begin{cases}
				a\left(u^n\right)^2 \frac{\mathcal{E}_{0} + \mathcal{E}_{1}}{2} \leq a\left(w_{M+1}^n\right)^2 \frac{\mathcal{E}_{M} + \mathcal{E}_{M+1}}{2}\,, & a>0\\
				a\left(u^n\right)^2\frac{\mathcal{E}_M + \mathcal{E}_{M+1}}{2} \geq  a\left(w_0^n\right)^2\frac{\mathcal{E}_0 + \mathcal{E}_1}{2}\,, & a<0
			\end{cases} .
		\end{equation}
		Then the discrete Lyapunov function is bounded by
		\begin{equation}
		\label{discr_decay_rate}
			\mathcal{L}^n \leq \operatorname*{exp}(-C_L n\Delta t)\mathcal{L}^0 + \Delta t\sum_{i=1}^n \operatorname*{exp}(-C_L (i-1)\Delta t) \mathcal{R}^{n-i}
		\end{equation}
		for a residual $\mathcal{R}^n = \mathcal{R}^n(\Delta t, \Delta x; a,q)$.
	\end{theorem}

	\begin{proof}
		For the proof  we mimic the steps in the proof of Lemma \ref{la:decay_pde_with_diffusion} for the partial differential equation with diffusion in the discrete case.
For better readability, we divide the proof into four parts.

\textbf{Part I.} For some fixed $n\geq 0$ the discrete derivative of the Lyapunov function \eqref{disc_Lyapunov} is given by
\begin{equation}
\label{before_central_diff}
  \frac{\mathcal{L}^{n+1}-\mathcal{L}^n}{\Delta t}
  =
  \sum_{i=1}^M   \frac{\left(w_i^{n+1}\right)^2 - \left(w_i^n\right)^2}{\Delta t} \mathcal{E}_i\Delta x =
  -\sum_{i=1}^M 2a w_i^n\frac{w_{i+1}^n - w_{i-1}^n}{2\Delta x}\mathcal{E}_i\Delta x	+ R_1^n(\lambda, a, q)
\end{equation}
with residual
			\begin{equation}\label{RestI}
				\begin{aligned}
					R_1^n(\lambda, a, q) &:= \frac{1}{\lambda} \sum_{i=1}^M q w_i^n\left(w_{i+1}^n-2w_i^n+w_{i-1}^n\right)\mathcal{E}_i\\
					&\quad + \frac{1}{\lambda} \sum_{i=1}^M \left(- \frac{\lambda a}{2}\left(w_{i+1}^n - w_{i-1}^n\right) + \frac{q}{2}\left(w_{i+1}^n-2w_i^n+w_{i-1}^n\right)\right)^2\mathcal{E}_i.
				\end{aligned}
			\end{equation}
Here we make use of \eqref{3point_scheme} that provides us with
			\begin{equation*}
				\begin{aligned}
					\left(w_i^{n+1}\right)^2 - \left(w_i^n\right)^2
					&= - \lambda a w_i^n\left(w_{i+1}^n - w_{i-1}^n\right)
					+ q w_i^n\left(w_{i+1}^n-2w_i^n+w_{i-1}^n\right)\\
					&\quad + \left(- \frac{\lambda a}{2}\left(w_{i+1}^n - w_{i-1}^n\right) + \frac{q}{2}\left(w_{i+1}^n-2w_i^n+w_{i-1}^n\right)\right)^2.
				\end{aligned}
			\end{equation*}

\textbf{Part II.} Inspired by the proof of Lemma \ref{la:decay_pde_with_diffusion} we  derive a discrete analogue to the continuous derivative $\left(aw^2\mathcal{E}\right)_x$.
For this purpose, we replace this derivative by the central difference
\begin{equation}
  \label{D_i}
  \mathcal{D}_i :=
  \frac{a\left(w_{i+1}^n\right)^2\mathcal{E}_{i+1} - a\left(w_{i-1}^n\right)^2\mathcal{E}_{i-1}}{2\Delta x}  .
\end{equation}
This term can be rewritten as
\begin{equation}
  \label{cont_deriv_as_central_diff}
    \mathcal{D}_i
    = aw_{i+1}^n \frac{w_{i+1}^n\mathcal{E}_{i+1} - w_i^n\mathcal{E}_i}{2\Delta x}  +
       a w_i^n\mathcal{E}_i\frac{w_{i+1}^n - w_{i-1}^n}{2\Delta x} +
       aw_{i-1}^n \frac{w_i^n\mathcal{E}_i - w_{i-1}^n\mathcal{E}_{i-1}}{2\Delta x} .
\end{equation}
Replacing $w_{i+1}^n = w_i^n + \left(w_{i+1}^n - w_i^n\right)$ and $w_{i-1}^n = w_i^n + \left(w_{i-1}^n - w_i^n\right)$  in the first and third term we obtain after some rearrangement
			\begin{equation*}
				\begin{aligned}
					\mathcal{D}_i	&= a\left(w_{i+1}^n\right)^2 \frac{\mathcal{E}_{i+1} - \mathcal{E}_i}{2\Delta x} + a\left(w_{i-1}^n\right)^2 \frac{\mathcal{E}_i - \mathcal{E}_{i-1}}{2\Delta x}
					+ 2 a w_i^n\mathcal{E}_i\frac{w_{i+1}^n - w_{i-1}^n}{2\Delta x}\\
					&\quad + a\mathcal{E}_i\frac{\left(w_{i+1}^n - w_i^n\right)^2}{2\Delta x} - a\mathcal{E}_i \frac{\left(w_i^n - w_{i-1}^n\right)^2}{2\Delta x}.
				\end{aligned}
			\end{equation*}
			Summation over all cells and using \eqref{D_i} yields
			\begin{equation}\label{will_be_replaced}
				\sum_{i=1}^M 2a w_i^n\mathcal{E}_i\frac{w_{i+1}^n - w_{i-1}^n}{2\Delta x} = \sum_{i=1}^M \frac{a\left(w_{i+1}^n\right)^2\mathcal{E}_{i+1} - a\left(w_{i-1}^n\right)^2\mathcal{E}_{i-1}}{2\Delta x}
				- \frac{1}{\Delta x} R_{\mathcal{E}}^n(a)	- \frac{1}{\Delta x} R_{2}^n(a)
			\end{equation}
			with residuals
			\begin{align}
			\label{R_E}
				R_{\mathcal{E}}^n(a) &:= \frac{a}{2}
				\sum_{i=1}^M \left(\left(w_{i+1}^n\right)^2 \left(\mathcal{E}_{i+1} - \mathcal{E}_i\right) + \left(w_{i-1}^n\right)^2 \left(\mathcal{E}_i - \mathcal{E}_{i-1}\right)\right),\\
			\label{R_2_n}
				R_{2}^n(a) &:= \frac{a}{2}	\sum_{i=1}^M \mathcal{E}_i \left( \left(w_{i+1}^n-w_i^n\right)^2 - \left(w_i^n-w_{i-1}^n\right)^2\right).
			\end{align}
			Inserting (\ref{will_be_replaced}) in (\ref{before_central_diff}) gives
			\begin{align*}
					&\frac{\mathcal{L}^{n+1}-\mathcal{L}^n}{\Delta t} =\\
					&- \left(\sum_{i=1}^M \frac{\left(w_{i+1}^n\right)^2 a\mathcal{E}_{i+1} - \left(w_{i-1}^n\right)^2a\mathcal{E}_{i-1}}{2\Delta x}
					- \frac{1}{\Delta x}R_{\mathcal{E}}^n(a) - \frac{1}{\Delta x}R_{2}^n(a)\right)\Delta x
					+ R_1^n(\Delta t, \Delta x; a, q).
			\end{align*}
By an index shift and the boundary conditions \eqref{Bed_BC_discr} the sum reduces to
			\begin{equation*}
				\begin{aligned}
					&\sum_{i=1}^M \frac{\left(w_{i+1}^n\right)^2 a\mathcal{E}_{i+1} - \left(w_{i-1}^n\right)^2a\mathcal{E}_{i-1}}{2\Delta x}\Delta x
					=  \\
					&\frac{a}{2} \left(-\left(w_{0}^n\right)^2\mathcal{E}_{0} - \left(w_{1}^n\right)^2\mathcal{E}_{1} + \left(w_{M}^n\right)^2\mathcal{E}_{M} + \left(w_{M+1}^n\right)^2\mathcal{E}_{M+1}\right)= \\
					& -\frac{a}{2}
					\begin{cases}
						\left(u^n\right)^2\left(\mathcal{E}_{0}+\mathcal{E}_1\right) + \left(\left(w_{1}^n\right)^2-\left(u^n\right)^2\right)\mathcal{E}_{1} - \left(w_{M+1}^n\right)^2\left(\mathcal{E}_{M} + \mathcal{E}_{M+1}\right)\,, &  a>0\\
						\left(w_{0}^n\right)^2\left(\mathcal{E}_{0} + \mathcal{E}_{1}\right) + \left(\left(u^n\right)^2-\left(w_{M}^n\right)^2\right)\mathcal{E}_{M} - \left(u^n\right)^2\left(\mathcal{E}_M + \mathcal{E}_{M+1}\right)\,, &  a<0
					\end{cases} .
				\end{aligned}
			\end{equation*}
The constraints \eqref{Bed_u_discr} on the control $u^n$ provides the inequality
			\begin{equation*}
				-\sum_{i=1}^M \frac{\left(w_{i+1}^n\right)^2 a\mathcal{E}_{i+1} - \left(w_{i-1}^n\right)^2a\mathcal{E}_{i-1}}{2\Delta x}\Delta x \leq R_u^n(a),
			\end{equation*}
where the residual on the right-hand side is determined by
			\begin{equation}\label{R_u}
				R_u^n(a) = \frac{a}{2} \begin{cases}
					\left(\left(w_{1}^n\right)^2-\left(u^n\right)^2\right)\mathcal{E}_{1}\,, & a>0\\
					\left(\left(u^n\right)^2-\left(w_{M}^n\right)^2\right)\mathcal{E}_{M}\,,& a<0
				\end{cases} .
			\end{equation}
Finally, the discrete time derivative of the discrete Lyapunov function can be estimated by
			\begin{equation}
			\label{eq:part2-1}
				\frac{\mathcal{L}^{n+1}-\mathcal{L}^n}{\Delta t} \leq R_{\mathcal{E}}^n(a) + R_u^n(a) + R_{2}^n(a) + R_1^n(\Delta t, \Delta x; a, q).
			\end{equation}

\textbf{Part III.} Now we consider the term $R_{\mathcal{E}}^n(a)$  given by (\ref{R_E}).
			We start by replacing the forward and the backward difference of $\mathcal{E}_i = \operatorname*{exp}(\mu(x_i))$ by the first derivative of the function $\mathcal{E}(x):=\mathcal{E}(\mu(x)) = \operatorname*{exp}(\mu(x))$. By Taylor expansion we obtain
			\begin{align*}
				&\mathcal{E}_{i+1} - \mathcal{E}_i = \mathcal{E}_{i+1}'\Delta x - \mathcal{E}_{i+1}''\frac{\Delta x^2}{2} + \mathcal{E}'''(\xi_i^+)\frac{\Delta x^3}{6},\\
				&\mathcal{E}_i - \mathcal{E}_{i-1} = \mathcal{E}_{i-1}'\Delta x + \mathcal{E}_{i-1}''\frac{\Delta x^2}{2} + \mathcal{E}'''(\xi_i^-)\frac{\Delta x^3}{6}
			\end{align*}
for  $\xi_i^+\in[x_i, x_{i+1}]$ and  $\xi_i^-\in[x_{i-1}, x_{i}]$, where we use the notation $\mathcal{E}_j^{(k)} = \mathcal{E}^{(k)}(x_j)$.
			Inserting these in (\ref{R_E}) gives
			\begin{align}
			\label{eq:part3-1}
				R_{\mathcal{E}}^n(a) = R_{\mathcal{E},0}^n(a)\Delta x + R_{\mathcal{E},1}^n(a)\Delta x^2 + R_{\mathcal{E},2}^n(a)\Delta x^3
			\end{align}
			with the three terms
			\begin{equation*}
				\begin{aligned}
					R_{\mathcal{E},0}^n(a) &= \sum_{i=1}^M \frac{a}{2}\left(w_{i+1}^n\right)^2\mathcal{E}_{i+1}'	+ \sum_{i=1}^M \frac{a}{2}\left(w_{i-1}^n\right)^2\mathcal{E}_{i-1}',\\
					R_{\mathcal{E},1}^n(a) &= -\sum_{i=1}^M \frac{a}{4}\left(w_{i+1}^n\right)^2\mathcal{E}_{i+1}'' + \sum_{i=1}^M \frac{a}{4}\left(w_{i-1}^n\right)^2\mathcal{E}_{i-1}'' ,\\
				R_{\mathcal{E},2}^n(a) &= \sum_{i=1}^M \frac{a}{12}\left( \left(w_{i+1}^n\right)^2\mathcal{E}'''(\xi_i^+) + \left(w_{i-1}^n\right)^2\mathcal{E}'''(\xi_i^-) \right).
				\end{aligned}
			\end{equation*}
			In the following we rewrite these terms. We start by considering the first term.
			By using the derivative
			\begin{equation}\label{der_E_j}
				\mathcal{E}_j' = \frac{d}{dx} \operatorname*{exp}(\mu(x)) \big|_{x_j} = \mu_x(x_j)\operatorname*{exp}(\mu(x_j)) \overset{(\ref{discr_Bed_mu})}{=} -\frac{C_L}{a}\mathcal{E}_j
			\end{equation}
and performing an index shift we  obtain with the definition of the discrete Lyapunov function \eqref{disc_Lyapunov}
			\begin{equation}
			\label{eq:part3-2}
				R_{\mathcal{E},0}^n(a)\Delta x = -C_L \mathcal{L}^n + R_0^n\Delta x
			\end{equation}
			with the residual
			\begin{equation}\label{R_0}
				R_0^n := - \frac{1}{2}C_L\left(\left(w_0^n\right)^2\mathcal{E}_0 - \left(w_1^n\right)^2\mathcal{E}_1 - \left(w_{M}^n\right)^2\mathcal{E}_{M} + \left(w_{M+1}^n\right)^2\mathcal{E}_{M+1}\right).
			\end{equation}
			For the second term we proceed similarly.
			Using
			\begin{equation}\label{der_der_E_j}
				\mathcal{E}_j'' = \left(\mathcal{E}_j'\right)'
				\overset{(\ref{der_E_j})}{=} \left(-\frac{C_L}{a}\mathcal{E}_j\right)'
				= -\frac{C_L}{a} \mathcal{E}_j'
				\overset{(\ref{der_E_j})}{=} \left(\frac{C_L}{a}\right)^2 \mathcal{E}_j,
			\end{equation}
and performing  an index shift we get
			\begin{equation}\label{R_E_1}
				R_{\mathcal{E},1}^n(a) = \frac{C_L^2}{4a} \left(\left(w_0^n\right)^2\mathcal{E}_0 + \left(w_1^n\right)^2\mathcal{E}_1 - \left(w_{M}^n\right)^2\mathcal{E}_{M} - \left(w_{M+1}^n\right)^2\mathcal{E}_{M+1}\right).
			\end{equation}
			In the third term we replace the third derivate using \eqref{der_E_j} and \eqref{der_der_E_j} by
			\begin{equation*}
				\mathcal{E}'''(x) = \left(\mathcal{E}''(x)\right)' = \left(\frac{C_L}{a}\right)^2 \mathcal{E}'(x) = -\left(\frac{C_L}{a}\right)^3 \mathcal{E}(x)
			\end{equation*}
			to obtain
			\begin{equation}\label{R_E_2}
				R_{\mathcal{E},2}^n(a) = - \frac{1}{12} \frac{C_L^3}{a^2} \sum_{i=1}^M \left( \left(w_{i+1}^n\right)^2\mathcal{E}(\xi_i^+) + \left(w_{i-1}^n\right)^2\mathcal{E}(\xi_i^-) \right).
			\end{equation}

			Combining \eqref{eq:part2-1}, \eqref{eq:part3-1} and \eqref{eq:part3-2} gives the decay
			\begin{equation}
			\label{eq:part3-3}
					\frac{\mathcal{L}^{n+1}-\mathcal{L}^n}{\Delta t} \leq
					-C_L \mathcal{L}^n + \mathcal{R}^n
			\end{equation}
			with the residual at time $t_n$ given by
			\begin{equation}\label{residual_R_n}
				\mathcal{R}^n := R_0^n\Delta x + R_{\mathcal{E},1}^n(a)\Delta x^2 + R_{\mathcal{E},2}^n(a)\Delta x^3 + R_u^n(a)+ R_{2}^n(a) + R_1^n(\lambda, a, q),
			\end{equation}
			where $R_0^n$ is defined by \eqref{R_0}, $R_{\mathcal{E},1}^n(a)$ by  \eqref{R_E_1}, $R_{\mathcal{E},2}^n(a)$ by \eqref{R_E_2}, $R_u^n(a)$ by \eqref{R_u}, $R_{2}^n(a)$ defined by \eqref{R_2_n} is simplified to
			\begin{equation}\label{R_2_n_II}
					R_{2}^n(a)= \frac{a}{2} \sum_{i=1}^M \mathcal{E}_i \left(w_{i+1}^n - w_{i-1}^n\right) \left(w_{i+1}^n - 2w_{i}^n + w_{i-1}^n\right)
			\end{equation}
			and $R_1^n(\lambda, a, q)$ is defined by  \eqref{RestI}.

\textbf{Part IV.}
Finally, we conclude from \eqref{eq:part3-3}
			\begin{equation*}
				\mathcal{L}^n \leq (1 -C_L\Delta t)\mathcal{L}^{n-1} + \Delta t \mathcal{R}^{n-1}.
			\end{equation*}
			Since this inequality holds for all $n\geq 1$ and since
$\left(1-C_L\Delta t\right)>0$ holds
by assumption \eqref{Bed_delta_t}, we can proceed iteratively to reach
			\begin{equation}\label{discrAbklingverhalten_ohneEXP}
				\mathcal{L}^n \leq (1 -C_L\Delta t)^n \mathcal{L}^0 + \Delta t\sum_{i=1}^n (1 -C_L\Delta t)^{i-1} \mathcal{R}^{n-i}.
			\end{equation}
Since
$\operatorname*{exp}(x) \geq 1+x$ for $x\in\R$  the assertion \eqref{discr_decay_rate} follows.

	\end{proof}

	\begin{remark}\label{equality_in_discr_rate}
		If we choose the discrete control $u^n$ such that equality holds in \eqref{Bed_u_discr}, then we also obtain
equality at least in the discrete decay rate \eqref{discrAbklingverhalten_ohneEXP}.
	\end{remark}

\subsection{Multi--dimensional analysis}
\label{subsec:Analysis-multiD}

By means of the 1D result we now prove the multi--dimensional result where we apply Theorem \ref{decay_discr_Lyapunov_theorem} along each direction of
the splitting scheme \eqref{eq:FVS-split}.

\begin{proof}[Proof of Theorem \ref{thm:decay_discr_Lyapunov_theorem-multid}]
In a first step we estimate the decay rate performing time evolution for the directions $k=1,\ldots,d$ of one time step with the splitted scheme \eqref{eq:FVS-split}. According to \eqref{eq:disc_Lyapunov-multid} the corresponding discrete Lyapunov function can be written as
\begin{align}
  \mathcal{L}^{n,k} =
  \sum_{\ibj\in\calJ} \left(w_\ibj^{n,k}\right)^2\mathcal{E}_\ibj |V_\ibj| =
  \sum_{\ibj|_k\in\calJ|_k} \frac{|V|}{\Delta x_k} \sum_{j_k\in\calJ_k}  \left(w_\ibj^{n,k}\right)^2\mathcal{E}(\bx_\ibj ) \Delta x_k .
\end{align}
Here we introduce the index set
\begin{align*}
  \calJ|_k :=\{ \bj|_k =(j_1,\ldots,j_{k-1},j_{k+1},\ldots,j_d) \in\R^{d-1} \,:\, j_l\in\calJ_l ,\ l\ne k   \},\
  \calJ_k:= \{ 1,\ldots,M_k\}.
\end{align*}
  We make the convention that $(\bj|_k,j_k) = \bj=(j_1,\ldots,j_d)$. Assuming  separation
 of variables \eqref{eq:E-mu} the Lyapunov function can  be rewritten as
 \begin{align}
 \label{eq:L_nk}
  \mathcal{L}^{n,k} =
  \sum_{\ibj|_k\in\calJ|_k}  \frac{|V|}{\Delta x_k}  \prod_{l=1,\ne k}^d  \mathcal{E}_l(x_l) \mathcal{L}_{\ibj|_k}^{n,k} ,\
  \mathcal{L}_{\ibj|k}^{n,k} :=
  \sum_{j_k\in\calJ_k}  \left(w_{\ibj}^{n,k}\right)^2\mathcal{E}_k(x_{j_k} ) \Delta x_k .
 \end{align}
 Applying Theorem \ref{decay_discr_Lyapunov_theorem}
with $C_L/d$ instead of $C_L$ in \eqref{discr_Bed_mu}
 to the direction
 characterized by $\bj|_k=const$ we estimate the decay of the discrete Lyapunov function corresponding to 
 this  by
 \eqref{eq:part3-3}
  resulting in
 \begin{align*}
   \mathcal{L}_{\ibj|k}^{n,k+1} \le
   \left(1 - \frac{C_L}{d} \Delta t\right) \mathcal{L}_{\ibj|_k}^{n,k} + \Delta t \mathcal{R}^{n,k}_{\ibj|_k}
 \end{align*}
where the residual in the $k^{th}$ intermediate step of the dimensional splitting scheme at  $\bj|_k=const$ is determined similarly to \eqref{residual_R_n} by
\begin{align}
\label{residual_R_n_k_j}
& \mathcal{R}^{n,k}_{\ibj|_k} :=
R^{n,k}_{\ibj|_k,0}\Delta x_k + R_{\ibj|_k,\mathcal{E},1}^{n,k}(a_k)\Delta x_k^2 + R_{\ibj|_k,\mathcal{E},2}^{n,k}(a_k)\Delta x_k^3 + \nonumber\\
&\hphantom{\mathcal{R}^{n,k}_{\ibj|_k} :=} R_{\ibj|_k,u}^{n,k}(a_k)+ R_{\ibj|_k,2}^{n,k}(a_k) + R_{\ibj|_k,1}^{n,k}(\lambda_k, a_k, q_k).
\end{align}
Here the residuals $R_{\ibj|_k,0}^{n,k}$,  $R_{\ibj|_k,\mathcal{E},1}^{n,k}(a_k)$, $R_{\ibj|_k,\mathcal{E},2}^{n,k}(a_k)$, $R_{\ibj|_k,u}^{n,k}(a_k)$, $R_{\ibj|_k,2}^{n,k}(\ibj|_k,a_k)$ and $R_{\ibj|_k,1}^{n,k}(\lambda_k, a_k, q_k)$
are  defined in analogy to
\eqref{R_0},   \eqref{R_E_1},  \eqref{R_E_2},   \eqref{R_u},  \eqref{R_2_n_II} and  \eqref{RestI}, respectively. Note that here the time step is $\Delta t$.
The approximation
$\mathcal{L}_{\ibj|k}^{n,k+1}$ and $\mathcal{L}_{\ibj|k}^{n,k}$ do not correspond to the times  $t_{n+(k+1)/d}$
and $t_{n+k/d}$, respectively.
Then, we conclude by definition
\eqref{eq:L_nk} for the total decay of the discrete Lyapunov function at time $t_{n+k/d}$
\begin{align*}
   \mathcal{L}^{n,k+1} \le
   \left(1 - \frac{C_L}{d} \Delta t \right) \mathcal{L}^{n,k} + \Delta t \mathcal{R}^{n,k}
 \end{align*}
 with the residual
\begin{equation}
\label{residual_R_n_k}
  \mathcal{R}^{n,k} :=
  \sum_{\ibj|_k\in\calJ_k}  \frac{|V|}{\Delta x_k}  \prod_{l=1,\ne k}^d  \mathcal{E}_l(x_l) \mathcal{R}_{\ibj|_k}^{n,k}
\end{equation}
Recursively applying the estimate for $k=d,\ldots,1$ yields
\begin{align*}
  \mathcal{L}^{n+1} &\equiv \mathcal{L}^{n,d} \le
   \left(1 - \frac{C_L}{d} \Delta t \right)^d \mathcal{L}^{n,0} + \Delta t \sum_{k=1}^d  (1 - C_L \Delta t /d)^{k-1} \mathcal{R}^{n,d-k} \\
   & \equiv \left(1 - \frac{C_L}{d} \Delta t \right)^d \mathcal{L}^{n} + \Delta t \mathcal{R}^{n} .
\end{align*}
Another similar recursion over the time steps $n$ provides
\begin{align*}
  \mathcal{L}^{n} \le
  \left(1 - \frac{C_L}{d} \Delta t \right)^{n d} \mathcal{L}^0 + \Delta t \sum_{i=1}^n  \left(1 - \frac{C_L}{d} \Delta t \right)^{d(i-1)} \mathcal{R}^{n-i} .
\end{align*}
From this the assertion \eqref{eq:discr_decay_rate-multid} follows.

\end{proof}

\section{Numerical results}
\label{sec:SimRes}

In this section we confirm numerically the influence of viscosity on the decay rate of the Lyapunov function in the one-dimensional case. We then extend the investigations to two dimensions in order to verify the multidimensional decay rate of Theorem \ref{thm:decay_discr_Lyapunov_theorem-multid}.
In this context, we investigate the control for our test cases and show how the control 
{enforces the decay of the solution.}

\subsection{Investigations in the 1D case}
\label{subsec:Analysis1D}

	To verify  numerically the decay rate \eqref{discr_decay_rate} of the discrete Lyapunov function \eqref{disc_Lyapunov} given in  Theorem \ref{decay_discr_Lyapunov_theorem},
	we consider
	IBVP \eqref{1D_IBVP}	
	in the domain $\Omega= (0, 1)$ using advection velocity  $a=2$.
	Simulations are run until the final time $T=3$.
	For initial condition we choose
	\begin{equation}\label{w0_1D}
		w_0(x) = \operatorname*{sin}(2\pi x).
	\end{equation}

	For an  admissible function $\mu$ in Theorem \ref{decay_discr_Lyapunov_theorem} satisfying
	\eqref{discr_Bed_mu}  we choose
	\begin{equation}\label{1D_mu_testeq}
		\mu(x) =  -\frac{C_L}{a} x,\qquad C_L=3.
	\end{equation}

	Since the advection velocity is positive, we observe $w_{M+1}^n$ in the ghost cell at the right boundary and apply the corresponding control $u^n = u^n\!\left(w_{M+1}^n\right)$ for each time step at the ghost cell in the left boundary by setting $w_0^n = u^n$, see \eqref{Bed_BC_discr}.
	The control $u^n$
	is chosen such that we have  equality
in \eqref{Bed_u_discr}, i.e., we set
	\begin{equation}\label{chosen_control_1D}
		u^n = \sqrt{\left(w_{M+1}^n\right)^2 \frac{\mE_M + \mE_{M+1}}{\mE_0 + \mE_1}}.
	\end{equation}
	This ensures also equality in the discrete decay rate \eqref{discrAbklingverhalten_ohneEXP} of Theorem \ref{decay_discr_Lyapunov_theorem}, i.e., the expected decay rate is predicted exactly, see Remark \ref{equality_in_discr_rate}.

	We run several computations for different 3--point schemes \eqref{3point_scheme} with
numerical viscosity coefficients $q\in\{ (\lambda a)^2, \lambda a, 1 \}$, respectively.	For the time step we choose $\dt = \operatorname*{CFL} \, a/\dx$ with $\operatorname*{CFL}=0.5$
	satisfying \eqref{Bed_delta_t} in Theorem \ref{decay_discr_Lyapunov_theorem}, i.e., $\Delta t = \Delta x$ and $\lambda = \Delta t /\Delta x =1$.

	 The decay of the resulting computed discrete Lyapunov function $\mathcal{L}^n$ is shown in Figure \ref{fig:disc_Lyapunov}.
	The decay rates are compared to the discrete decay rate without residual corresponding to the discretized exact decay rate of
	Theorem \ref{Hauptresultat} and of
	Lemma \ref{la:decay_pde_with_diffusion} in the case $\bar{q}=0$,
	i.e., without the additional diffusion,
	and to the decay rates in \eqref{discr_decay_rate} and \eqref{discrAbklingverhalten_ohneEXP} from Theorem \ref{decay_discr_Lyapunov_theorem} and its proof, respectively.
	For the implementation of the residual \eqref{R_E_2} in \eqref{discr_decay_rate} and \eqref{discrAbklingverhalten_ohneEXP} we use the upper bound
	\begin{equation*}
		R_{\mE,2}^n(a) \leq -\frac{1}{12}\frac{C_L^3}{a^2} \sum_{i=1}^M \left(\left(w_{i+1}^n\right)^2\mE_{i+1} + \left(w_{i-1}^n\right)^2\mE_{i}\right).
	\end{equation*}
	Note that $\mE(\mu(x))$ is monotonically decreasing in $x$ for the function $\mu$ determined by \eqref{1D_mu_testeq}.
	In the figures we plot only $10$ out of all time steps.

	\begin{figure}[htbp]
		\centering
		\includegraphics[trim = 1.0cm 14.5cm 2.0cm 7.5cm, clip,width=0.82\textwidth]{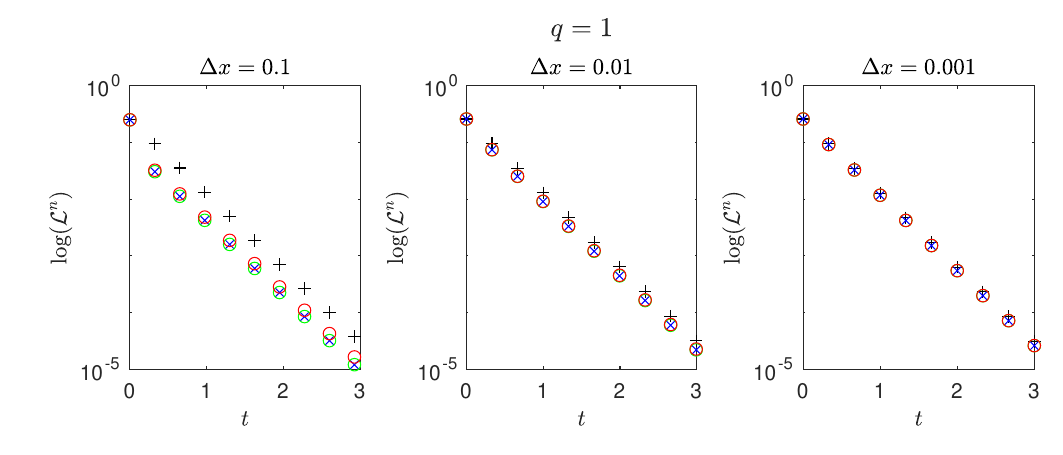}
		\includegraphics[trim = 1.0cm 14.5cm 2.0cm 7.5cm, clip,width=0.82\textwidth]{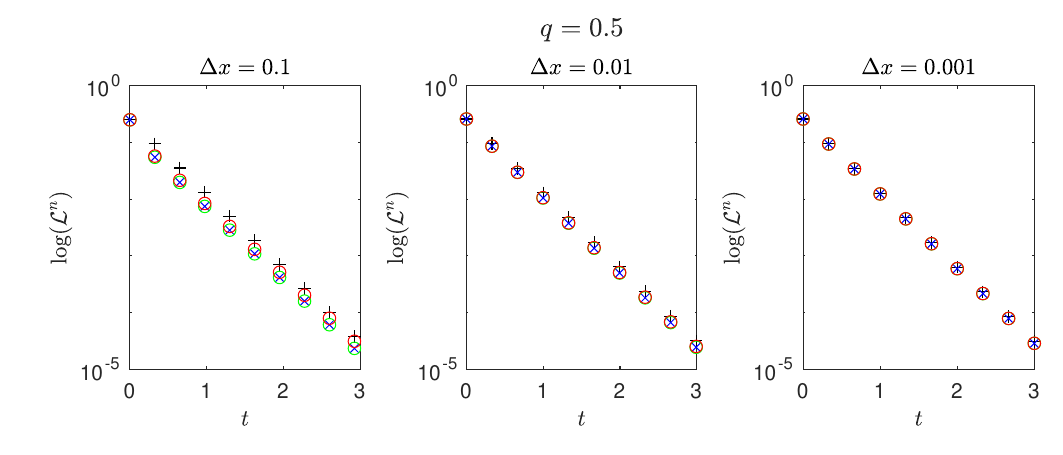}
		\includegraphics[trim = 1.0cm 8cm 2.0cm 7.5cm, clip,width=0.82\textwidth]{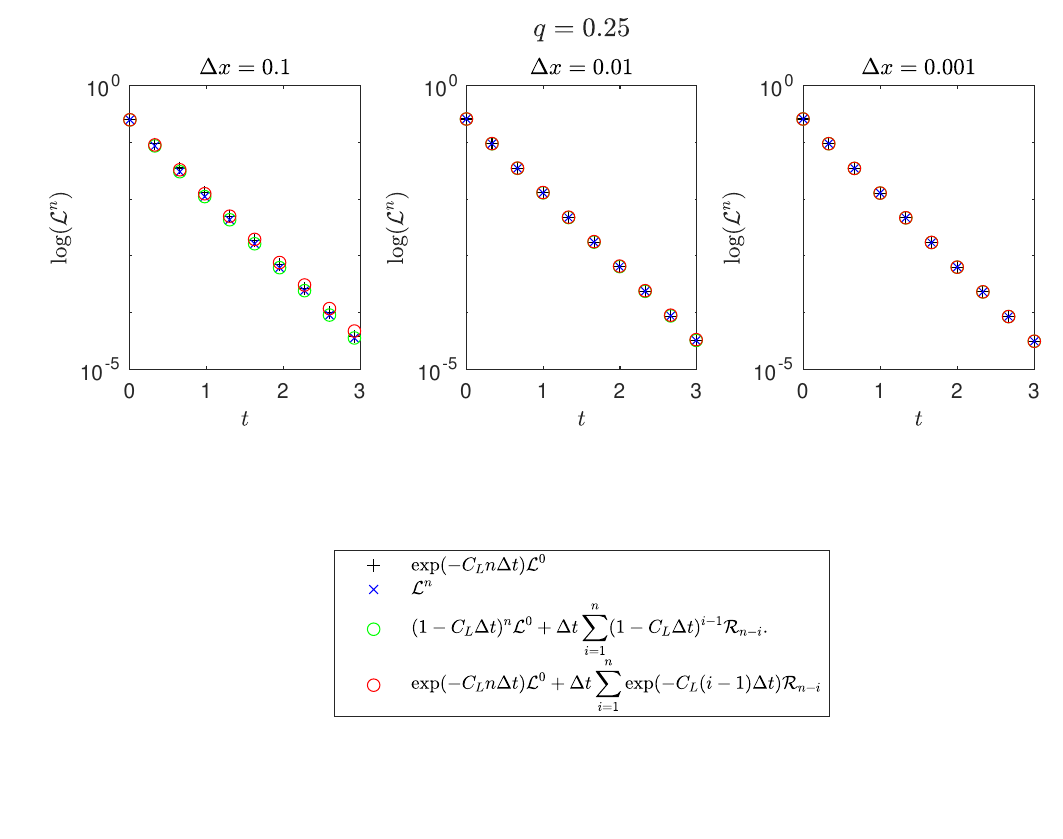}
		\caption{
			Decay of the discrete Lyapunov function $\mathcal{L}^n$ for schemes
with $q=1$ (top), $q=\lambda a=0.5$ (middle) and $q=(\lambda a)^2=0.25$ (bottom) and fixed $\operatorname*{CFL} = 0.5$ compared to the discretized decay rate without diffusion $\operatorname*{exp}(-C_L n\dt)\mathcal{L}^0$ and the discrete decay rates of Theorem \ref{decay_discr_Lyapunov_theorem}.
Results are shown for $10$ equidistantly distributed time steps in $[0,T]$.
}
		\label{fig:disc_Lyapunov}
	\end{figure}

\textbf{Decay rates.} 
We consider the decay rates of the different schemes.
	For $q=1$ we observe that the computed discrete Lyapunov function $\mathcal{L}^n$ converges for $\dx\to 0$ to the decay rate  $\operatorname*{exp}(-C_L n\dt)\mathcal{L}^0$
	as expected according to
	the statement at the end of Sec.~\ref{subsec:Analysis}.
	The decay rate in \eqref{discrAbklingverhalten_ohneEXP} is always a very good approximation of the computed discrete Lyapunov function.
	Since \eqref{discr_decay_rate}  follows from \eqref{discrAbklingverhalten_ohneEXP}, the  right-hand side in \eqref{discr_decay_rate}
	overestimates the computed decay rate of $\mathcal{L}^n$ slightly, but for $\dx\to 0$ this rate coincides with \eqref{discrAbklingverhalten_ohneEXP} and both converge to the rate $\operatorname*{exp}(-C_L n\dt)\mathcal{L}^0$.

	For  $q=\lambda a < 1$ the results are similar to $q=1$, except that the decay rates of $\mathcal{L}^n$, (\ref{discrAbklingverhalten_ohneEXP}) and \eqref{discr_decay_rate} are now much closer to $\operatorname*{exp}(-C_L n\dt)\mathcal{L}^0$ for all three choices of $\dx$.

	For $q = \left(\lambda a\right)^2$ and $\dx=0.1$ we can see that \eqref{discr_decay_rate} is even above $\operatorname*{exp}(-C_L n\dt)\mathcal{L}^0$.
	However, for $\dx\leq 0.01$, all rates coincide.
	We already expected this due to the results of Section
	\ref{sec:Proof-1D}.
In this case,
we have a second-order scheme for $q = (\lambda a)^2$.

	In summary, the rates
 \eqref{discrAbklingverhalten_ohneEXP} and also \eqref{discr_decay_rate} approximate the computed decay rate of the Lyapunov function much better than the analytical of
	Theorem \ref{Hauptresultat} and of
	Lemma \ref{la:decay_pde_with_diffusion} in the case $\bar{q}=0$.
	In addition, we found that the decay rate of the computed Lyapunov function is closer to $\operatorname*{exp}(-C_L n\dt)\mathcal{L}^0$ for a smaller numerical viscosity $q$.
	Therefore, diffusion
 causes the Lyapunov function to decay faster.
	If we decrease $q$ to the lower limit $q =\left(\lambda a\right)^2$ for which scheme %
\eqref{3point_scheme}
	is still stable, the rates even coincide for sufficiently small $\dx$.

	In Table \ref{tab:EoC_q} we compare the empirical order of convergence (EoC) of the discrete Lyapunov function at time 
$T=3$
for the refinement from $\dx=0.01$ to $\dx=0.001$. We obtain the largest EoC for $q = (\lambda a)^2$ and the smallest for $q=1$.
	Obviously, diffusion has a strong  influence on the order of convergence, because  for $q=\lambda a$ and $q=1$ the scheme is only  first--order accurate whereas it is second--order accurate for  $q = (\lambda a)^2$ .

	\begin{table}[htbp]
		\centering
		\begin{tabular}{ c||c|c|c }
			$q$ & $\left(\lambda a\right)^2 $ & $\lambda a$ & $1$\\
			\hline
			EoC & 0.8844 & 0.4640 & 0.2796 \\
		 \end{tabular}
		\caption{Computed empirical order of convergence of the discrete Lyapunov function for different $q$ at $T=3$.
}
		\label{tab:EoC_q}
	\end{table}

\textbf{Control and residual}
We investigate the discrete control $\left(u^n\right)_{n=0,\dots,N}$ and the residual $\left(\mathcal{R}_n\right)_{n=0,\dots,N}$.
	We expect the control $\left(u^n\right)_{n=0,\dots,N}$ to decrease over time.

	\begin{figure}[!htbp]
		\centering
		\subfloat[\centering $q=1$]{{\includegraphics[ width=0.48\textwidth]{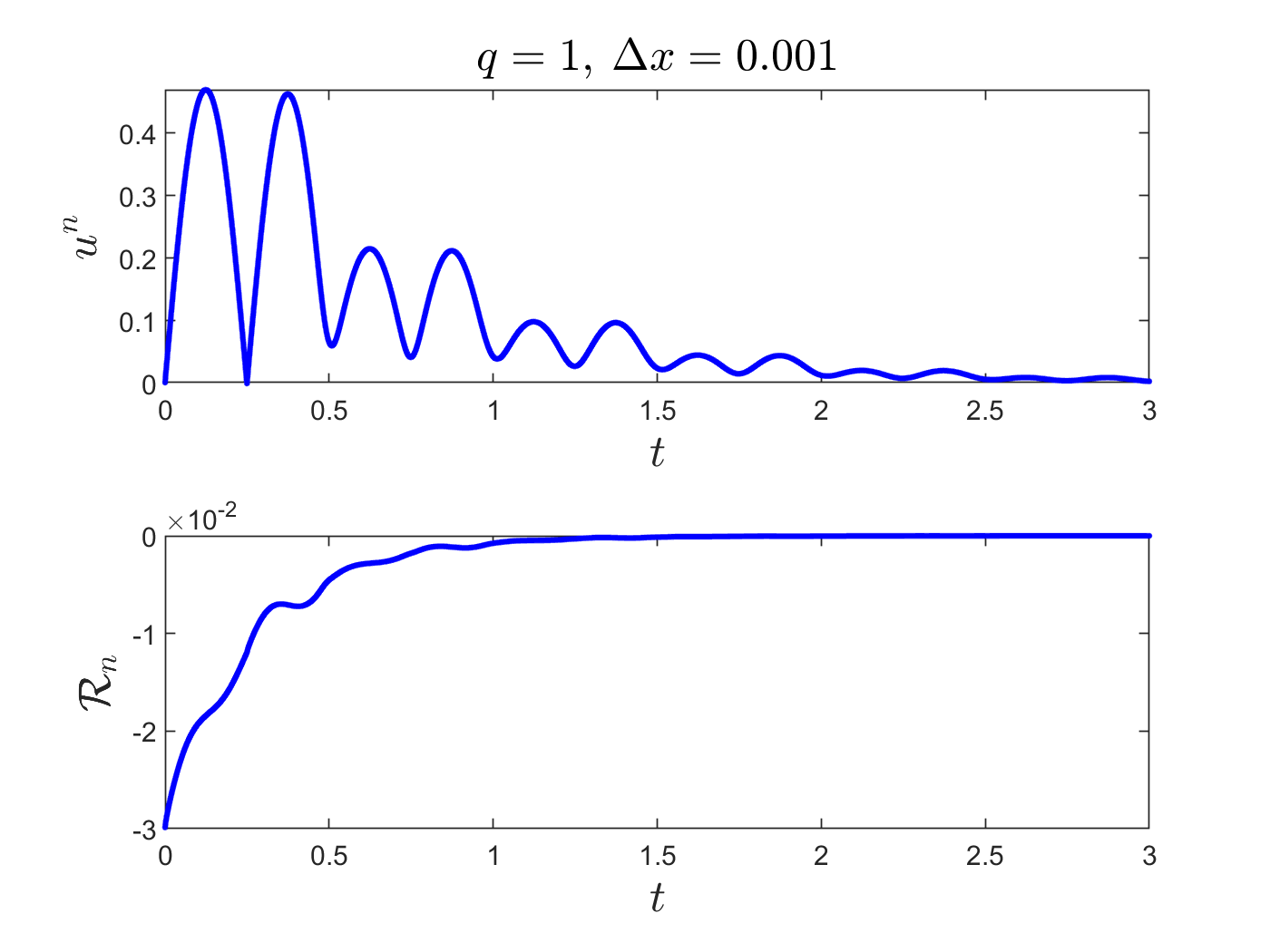} }}
		\
		\subfloat[\centering $q=\left(\lambda a\right)^2$]{{\includegraphics[width=0.48\textwidth]{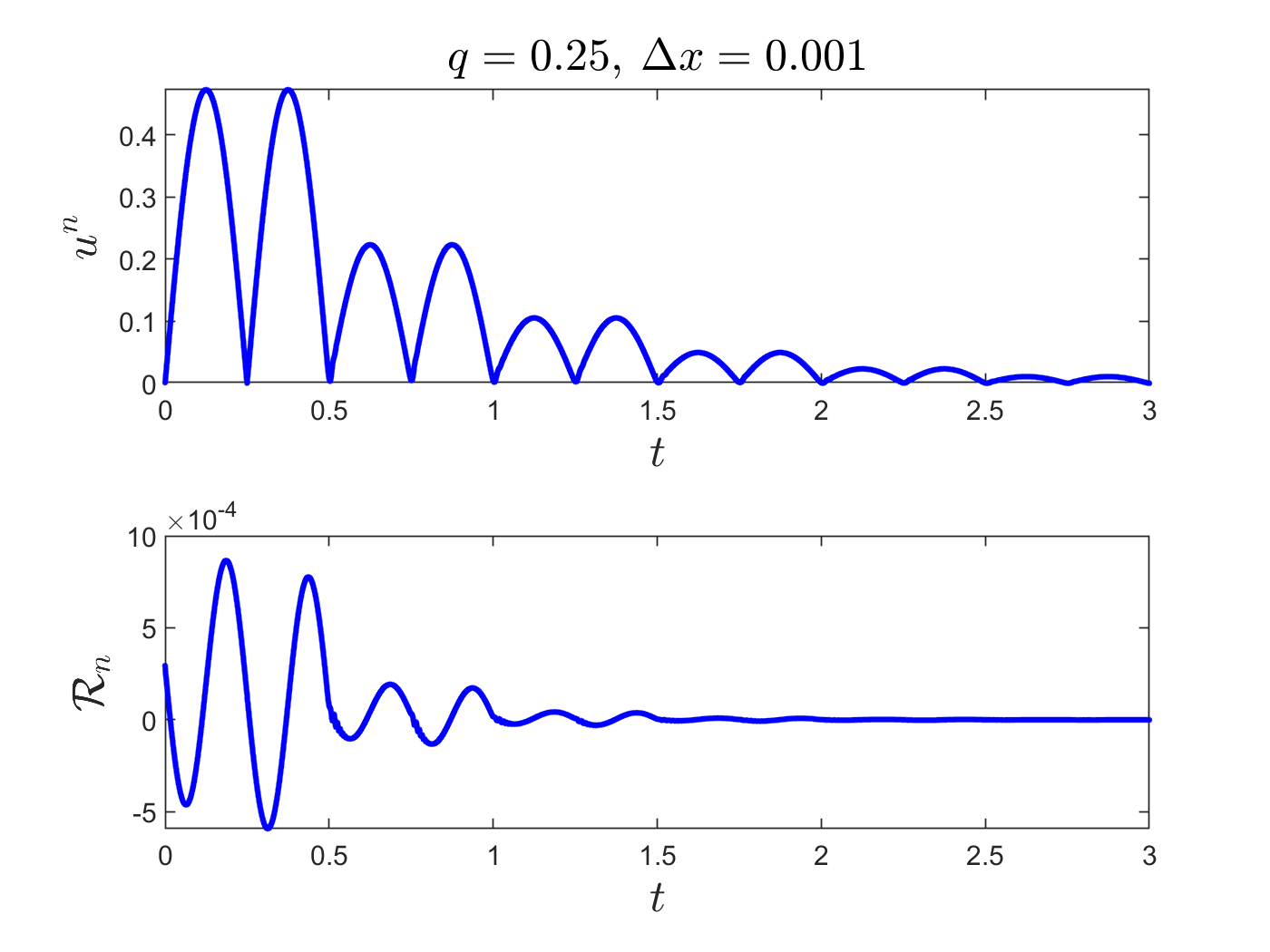} }}
		\caption{Plots of the discrete control $\left(u^n\right)_{n=0,\dots,N}$ and the residual $\left(\mathcal{R}_n\right)_{n=0,\dots,N}$ for different values of $q$
with $\operatorname*{CFL}=0.5$ and $\dx = 0.001$.
}
		\label{fig:disc_Lyapunov_u_Rn}
	\end{figure}

	We compare the above simulations  with the finest resolution, i.e., $\dx = 0.001$.
	Here, the control $\left(u^n\right)_{n=0,\dots,N}$ in Figure \ref{fig:disc_Lyapunov_u_Rn} can be explained by the initial data \eqref{w0_1D}, which is transported to the right boundary with velocity $a=2$.
	Based on the boundary values at this boundary, the control is chosen by \eqref{chosen_control_1D}.
	Note that $u^n$ is chosen as a positive value and
	\begin{equation*}
		\sqrt{\frac{\mE_M + \mE_{M+1}}{\mE_0+\mE_1}} < 1.
	\end{equation*}
	The resulting oscillations of $\left(u^n\right)_{n=0,\dots,N}$ can be observed particularly well by the plot for the second--order method with $q=\left(\lambda a\right)^2$, see the upper right plot in Figure \ref{fig:disc_Lyapunov_u_Rn}.

	The residual $\mathcal{R}_n$ also vanishes as 
$t_n\to T$.
	For the second--order scheme with $q=\left(\lambda a\right)^2$, $\left(\mathcal{R}_n\right)_{n=0,\dots,N}$ is smaller by two orders of magnitude than for the first--order scheme with $q=1$.

\subsection{Multi--d stabilization using higher order schemes}
\label{subsec:MultidHighOrder}

To perform higher order computations in the multi--dimensional case we adapted {the in--house C++--library \textit{MultiWave} to the
IBVP \eqref{eq:IBVP}.
This library provides a multi--dimensional high order Runge--Kutta discontinuous Galerkin scheme with grid adaptation that can be applied for solving hyperbolic balance laws, see \cite{multiwave}.

\textbf{Solver.} We briefly summarize the main ingredients of \textit{MultiWave}.
Here we apply a $P^{th}$--order DG scheme using  piecewise polynomial elements of order $P=1,2,3$
and a $P^{th}$--order explicit SSP--Runge--Kutta method with $P$ stages for the time discretization.
For a numerical flux we choose the local Lax-Friedrichs flux that for our linear fluxes coincides with the classical Lax--Friedrichs flux. 
Thus, in case of $P=1$ the scheme reduces to a 3-point scheme as \eqref{eq:FVS-split}. 
The Gibbs phenomena near to discontinuities is suppressed by the minmod limiter from \cite{Cockburn:1998jt}
for $P\ge 2$. 
Because of  the explicit time discretization we restrict the 
time step by means of a CFL condition.
The efficiency of the scheme is improved by local grid adaption where we employ the
multiresolution concept based on multiwavelets.
The key idea is to perform a multiresolution analysis on a sequence of
nested grids providing a decomposition of the data on a coarse scale and a sequence
of details that encode the difference of approximations on
subsequent resolution levels. The detail coefficients become small when the underlying
data are locally smooth and, hence, can be discarded when dropping below a level--dependent
threshold value $\varepsilon_l = 2^{l-\Lmax} c_{thresh} h_{\Lmax}$ where $\Lmax$ denotes the largest refinement level and $h_{\Lmax}$ the smallest cell diameter on this level. The parameter $c_{thresh}$ is chosen proportional to the smallest amplitude in the initial data.
By means of the thresholded sequence a new, locally refined grid
is determined. 
Details can be found in
\cite{HovhannisyanMuellerSchaefer-2014,GerhardIaconoMayMueller-2015,GerhardMueller-2016,GerhardMuellerSikstel:2021}.\\
For all examples, we set the CFL number to $CFL=0.7$. The
cell size  at the coarsest refinement level is 
$\Delta\mathbf{x}_0 =(1/3,1/3)$. 
Due to the dyadic grid hierarchy in \textit{MultiWave} the cell size at refinement level $l$ is $\Delta\mathbf{x}_l = 2^{-l}\Delta\mathbf{x}_0$, for $l=0,\dots,\Lmax$.

\textbf{Configuration.}
	In order to have
a direct comparison to the results reported in \cite{case_study}, we investigate a configuration from this reference.
	This allows  to verify the implementation of the boundary feedback control.

	For the bounded domain $\Omega = (0, 1)\times (0, 1) \subset\R^2$ we consider the  IBVP \eqref{eq:IBVP} with
	\begin{align*}
		\V{A}^{(1)} :=
		\begin{pmatrix}
			4 & 0 \\ 0 & 2
		\end{pmatrix}
		,\quad
		\V{A}^{(2)} :=
		\begin{pmatrix}
			2 & 0 \\ 0 & -2
		\end{pmatrix}
		\quad\text{and}\quad \V{B}\equiv\V{0}.
	\end{align*}
To describe the inflow boundaries and corresponding boundary conditions we introduce
	 $\V{a}_i := (a_{ii}^{(1)}, a_{ii}^{(2)})$, i.e.,
	we have $\V{a}_1 = (4, 2)$ and $\V{a}_2 = (2,-2)$.
	By the inner products $\V{a}_i(\V{x})\cdot\V{n}(\V{x})$ at the boundary, we  partition the boundary $\partial\Omega$ into the subsets
	\begin{align*}
		& \Gamma_1^- = \{0\}\times [0, 1]\cup [0, 1]\times\{0\} \quad\text{and}\quad \Gamma_1^+ = \{1\}\times [0, 1] \cup [0, 1]\times\{1\},\\
		&\Gamma_2^- = \{0\}\times [0, 1]\cup [0, 1]\times\{1\} \quad\text{and}\quad \Gamma_2^+ = [0, 1]\times\{0\}\cup\{1\}\times [0, 1]
	\end{align*}
	for the first component $w_1$ as well as
	for the second component $w_2$.
	As in \cite{case_study}, we divide $\Gamma_i^-$ into the subsets
	\begin{equation*}
		\mathcal{C}_1 = \mathcal{C}_2 = \{0\}\times [0, 1],\quad
		\mathcal{Z}_1 = [0, 1]\times\{0\} \quad\text{and}\quad \mathcal{Z}_2 = [0, 1]\times\{1\}.
	\end{equation*}
	At $\mathcal{C}_i$ we prescribe the boundary values by the control
	\begin{align}
	\label{integral_for_testcontrol}
		u(t) = \sqrt{\frac{1}{6 (e-1)} I(t)}, \quad
		I(t) = \sum_{i=1}^2 \int_{\Gamma_i^+} w_i^2 (\V{a}_i\cdot\V{n}) \operatorname*{exp}(\mu_i(\V{x}))\:d\boldsymbol{\sigma}.
	\end{align}
	This control can be computed from condition \eqref{eq:Bed_u} with ``$=$'' instead of ``$\le$'' under the restriction $u_i(t,x) = u(t)$ for $i\in\{1,2\}$.
	Following \cite{case_study}  we  use the functions
	\begin{equation*}
		\mu_1(\V{x}) = x_2 - \left(\frac{1}{2} + \frac{1}{4} C_L\right)x_1\quad\text{and}\quad \mu_2(\V{x}) = x_2 + \left(1 - \frac{1}{2} C_L\right)x_1
	\end{equation*}
	in \eqref{integral_for_testcontrol} determined by the condition \eqref{eq:Bed_mu}.
	In the computations  we fix $C_L = 3$.

	The initial conditions are given by the smooth function
	\begin{equation}\label{w0_paper}
		w_i(0, \V{x}) = \operatorname*{sin}(2 \pi x_1) \operatorname*{sin}(2 \pi x_2),\
		i=1,2.
	\end{equation}

\textbf{Decay rate.}
The adaptive DG solver provides piecewise polynomial data, i.e., on each element  in the adaptive grid corresponding to the index set $\calJ$ we have a polynomial of order $P$. For the discretization of the Lyapunov function \eqref{eq:Lyapunov} we apply a Gauss--Legendre quadrature rule on each cell with $r=(P+1)^d$ sample points, i.e.,
	\begin{align}
	\label{L_hat}
	\hat{L}(t) = 		\sum_{C\in\:\calJ} \frac{\left\lvert C\right\rvert}{4} \sum_{i_1=0}^P\left( \sum_{i_2=0}^P \left(\V{w}(t, (x_{i_1}^C, x_{i_2}^C))\right)^T\mE((x_{i_1}^C, x_{i_2}^C))\V{w}(t, (x_{i_1}^C, x_{i_2}^C))\alpha_{i_2} \right)\alpha_{i_1} \hat{L}(t).
	\end{align}

	We perform two computations with $P=1$ and $P=2$ using constant or linear polynomials. Note that for $P=1$ the DG scheme  is the classical
Lax-Friedrichs finite volume scheme. 	The resulting approximated Lyapunov functions $\hat{L}(t)$ are compared to the predicted decay rate $\exp(-C_L t) L(0)$ given in \eqref{eq:AbklingrateBeweis} in Figure \ref{fig:LyapunovP1_P2_opt}.
	\begin{figure}[htbp]
		\centering
		\includegraphics[width=0.4\textwidth]{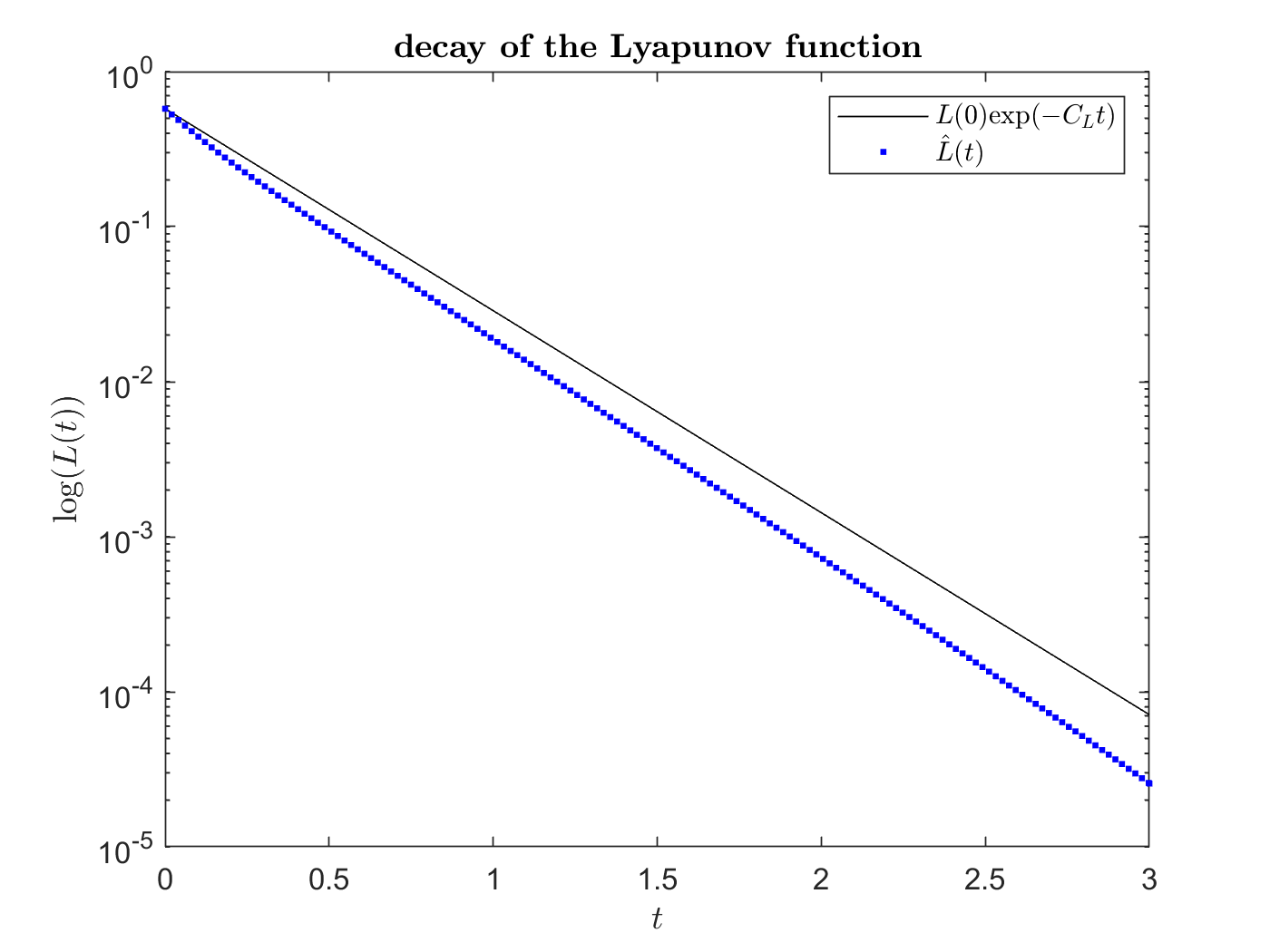}
		\includegraphics[width=0.4\textwidth]{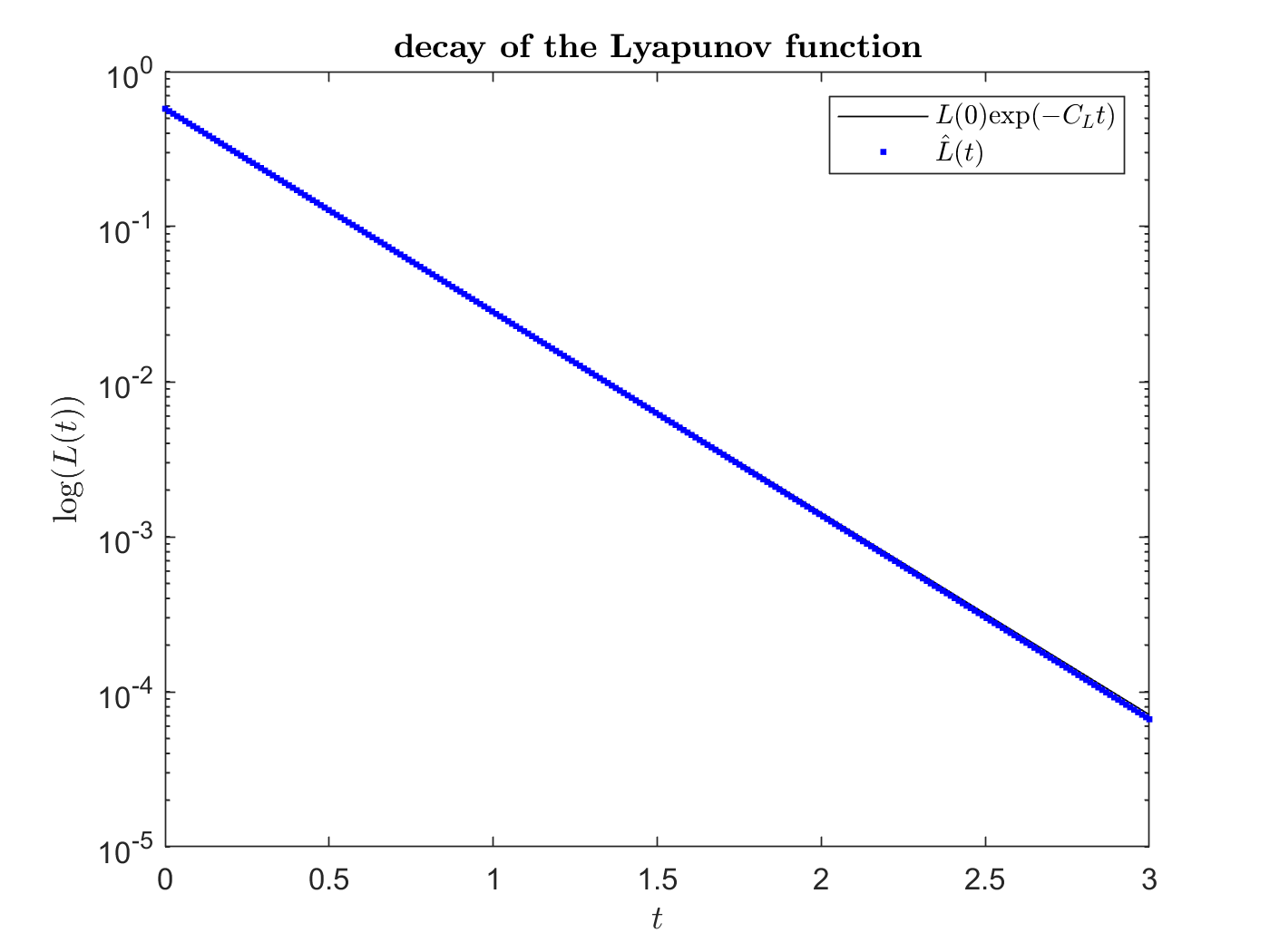}
		\caption{Decay of the approximated Lyapunov function $\hat{L}(t)$ for a simulation with $P=1$ (left) and $P=2$ (right), $\operatorname*{CFL} = 0.7$, $\Lmax=6$ and $c_{thresh} = 0.1$ compared to the exact decay rate $L(0)\operatorname*{exp}(-C_L t)$ in Theorem \ref{Hauptresultat}.}
		\label{fig:LyapunovP1_P2_opt}
	\end{figure}

	Our computations verify that the computed Lyapunov function decays faster than it is predicted analytically in Theorem \ref{Hauptresultat} due to higher numerical dissipation.
	Moreover,  the decay rate of the Lyapunov function for the second--order DG scheme
	is closer to the exact one than the one for the first--order  scheme. This confirms the findings in
	\cite{case_study} where the computations are performed using first--order and second--order finite volume schemes. This is in agreement with our analysis in Section
	\ref{subsec:Analysis} and corresponding numerical investigations in the Sections \ref{sec:Proof-1D} and	\ref{subsec:Analysis-multiD}
	and corresponding numerical results in Section \ref{subsec:Analysis1D}.

\textbf{Rate of convergence of the Lyapunov function.}
To investigate the order of convergence of $\hat{L}(t)$ to the exact rate
$L(0)\operatorname*{exp}(-C_L t)$ in Theorem \ref{Hauptresultat}
 we perform computations for varying  maximum refinement level $\Lmax \in\{3,\dots,8\}$ for different polynomial order $P\in\{1,2,3\}$. The empirical order of convergence is determined at the final time 
$T=3$.

In Figure \ref{fig:convergenceP1-P23} we plot  the approximated rates of convergence of 
$\hat{L}(T)$ 
for the different maximum refinement levels over the smallest possible area of an element on refinement level $l$  determined by
$A(l) := \left\lvert \Omega\right\rvert \left(\frac{1}{3\cdot 2^l}\right) ^2$.
For $P=1$ the asymptotic regime is reached only for higher refinement levels whereas for $P=2,3$, the asymptotic behavior can be observed also for lower refinement levels.
	The rates for $P=2$ and $P=3$ look similar, but the rate for $P=3$ seems to be a bit better.
	As can be expected, both rates are larger than in the case of $P=1$.

	\begin{figure}[htbp]
		\centering
		\includegraphics[width=0.5\textwidth]{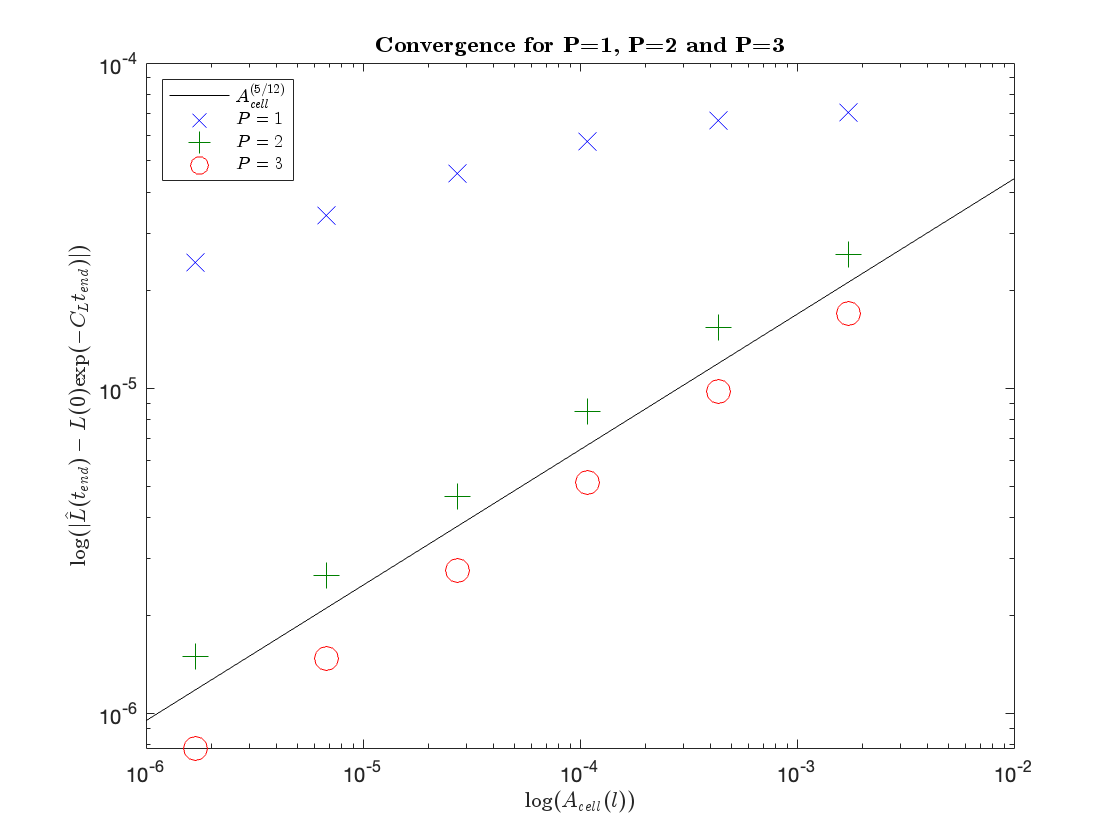}
		\caption{
			Convergence of the approximated Lyapunov function $\hat{L}(t)$ for different maximum refinement levels for simulations with $P=1$, $P=2$ and $P=3$ with $\operatorname*{CFL} = 0.7$ with grid adaptation ($c_{thresh} = 0.1$) to the exact one $L(0)\operatorname*{exp}(-C_L t)$ at 
$T=3$.
}
		\label{fig:convergenceP1-P23}
	\end{figure}

	Next we compute the empirical order of convergence (EoC) for the three cases.
	The EoC for levels $l$ and $l+1$ at time $t$ is given by
	\begin{equation}
	\label{OrderOfConvergence}
		p = \text{log}_{\frac{A(l)}{A(l+1)}} \frac{\left\lvert \hat{L}_{l}(t) - L(t) \right\rvert}{\left\lvert \hat{L}_{{l+1}}(t) - L(t) \right\rvert}=
		\text{log}_4 \frac{\left\lvert \hat{L}_{l}(t) - L(t) \right\rvert}{\left\lvert \hat{L}_{{l+1}}(t) - L(t) \right\rvert}
		= \frac{\operatorname*{ln}\frac{\left\lvert \hat{L}_{l}(t) - L(t) \right\rvert}{\left\lvert \hat{L}_{{l+1}}(t) - L(t) \right\rvert}}{\operatorname*{ln} 4}.
	\end{equation}
	Here $\hat{L}_{l}(t)$ denotes the approximated Lyapunov function $\hat{L}(t)$ at time $t$ for a simulation with a maximum refinement level $l$ and $L(t) = L(0)\operatorname*{exp}(-C_L t)$ is the value of the exact Lyapunov function given in Theorem \ref{Hauptresultat}.
	The resulting EoC for varying maximum refinement level is presented in Table \ref{tab:OrderOfConvergence}.

	\begin{table}[htbp]
		\centering
		\begin{tabular}{ c||c|c|c|c|c  }
			& $\Lmax=3 $ & $\Lmax=4$ & $\Lmax=5$ & $\Lmax=6$ & $\Lmax=7$\\
			\hline
			$P=1$ & 0.0397 & 0.1061 & 0.1665 & 0.2116 & 0.2382 \\
			$P=2$ & 0.3714 & 0.4311 & 0.4329 & 0.4036 & 0.4129 \\
			$P=3$ & 0.3984 & 0.4640 & 0.4465 & 0.4530 & 0.4561 \\
		 \end{tabular}
		\caption{Computed empirical order of convergence (\ref{OrderOfConvergence}) of the approximated Lyapunov function $\hat{L}(t)$ for different maximum refinement levels $l$ at 
$T=3$.
}
		\label{tab:OrderOfConvergence}
	\end{table}

	We note that the approximated Lyapunov function converges very slowly to the exact rate. Thus,  we need to significantly increase the maximum refinement level to get a good approximation of the exact Lyapunov function and to obtain a good approximation of the EoC.
	The EoC improves with increasing order $P$.

No regular dependence of $P$ on the rate of convergence can be observed.
This confirms the statement
at the end of Sec.~\ref{subsec:Analysis}.
	On the one hand, the numerical scheme introduces a discretization error and a projection error, since we approximate the solution with functions from the DG space.
	On the other hand, we only consider the approximated Lyapunov function $\hat{L}(t)$ instead of \eqref{eq:Lyapunov} introducing a quadrature error.
	These errors combined might dominate.

\textbf{Boundary feedback control.}
	To understand the behavior of the boundary feedback control
	we consider the computed solution $\V{w}$
	at time $t\in[0.0, 0.6]$.
	We know that the decay of the Lyapunov function guarantees that the solution decays.
	\begin{figure}[htbp]
		\centering
		\includegraphics[trim = 0cm 5cm 0cm 0cm, clip, width=0.83\textwidth]{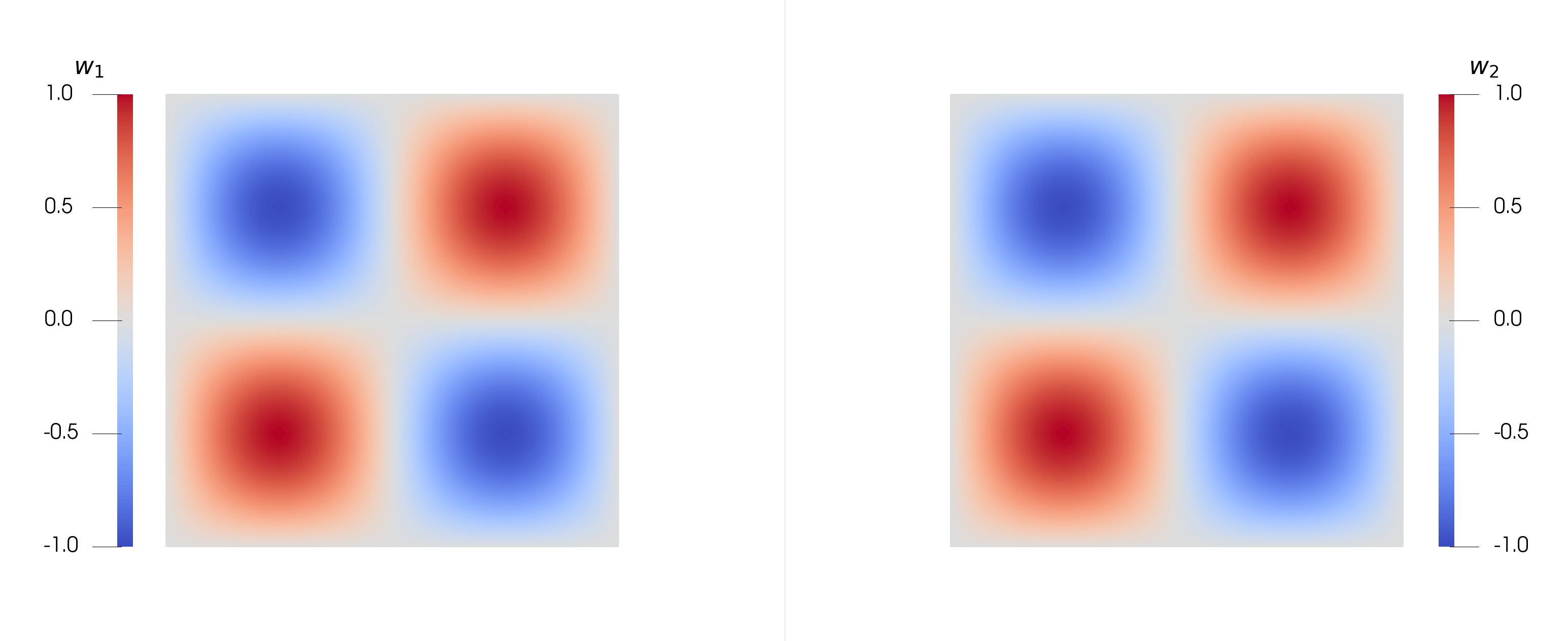}
		\caption*{$t=0.0$}
		\vspace{0.2cm}
		\includegraphics[trim = 0cm 5cm 0cm 0cm, clip, width=0.83\textwidth]{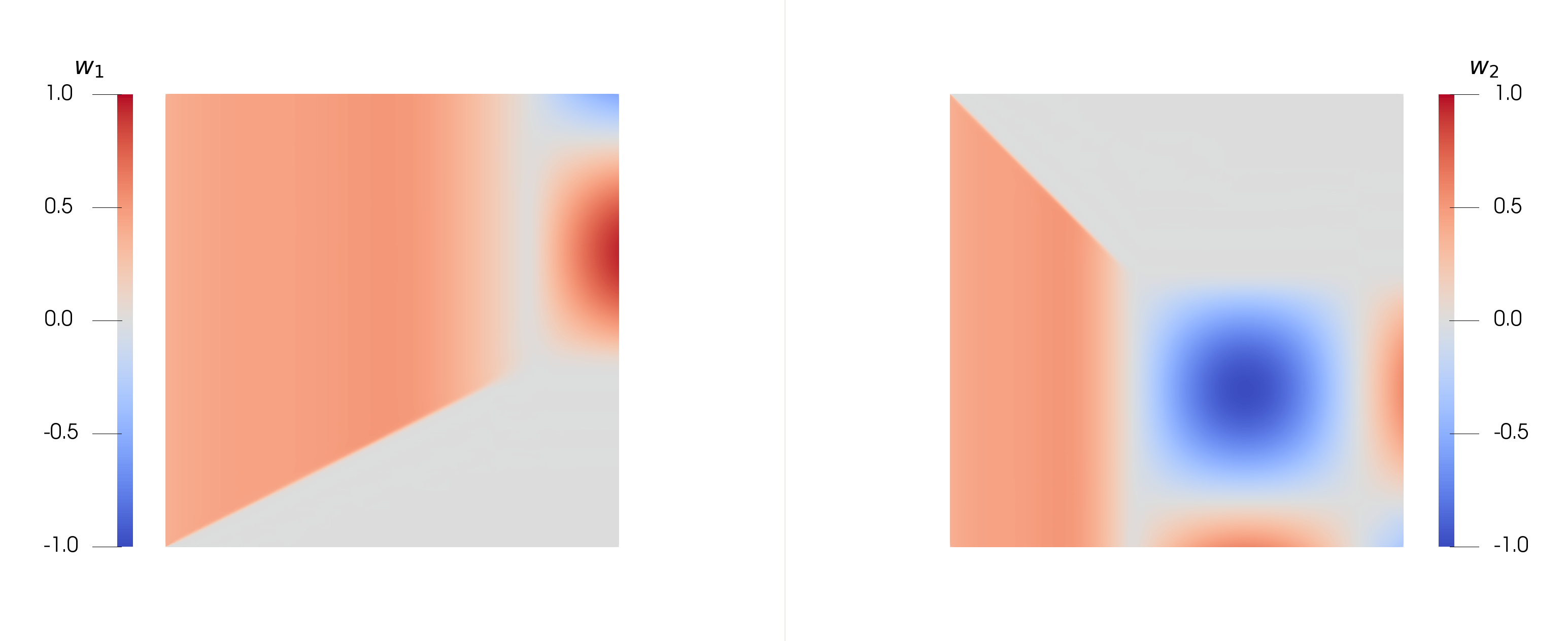}
		\caption*{$t=0.2$}
		\vspace{0.2cm}
		\includegraphics[trim = 0cm 5cm 0cm 0cm, clip, width=0.83\textwidth]{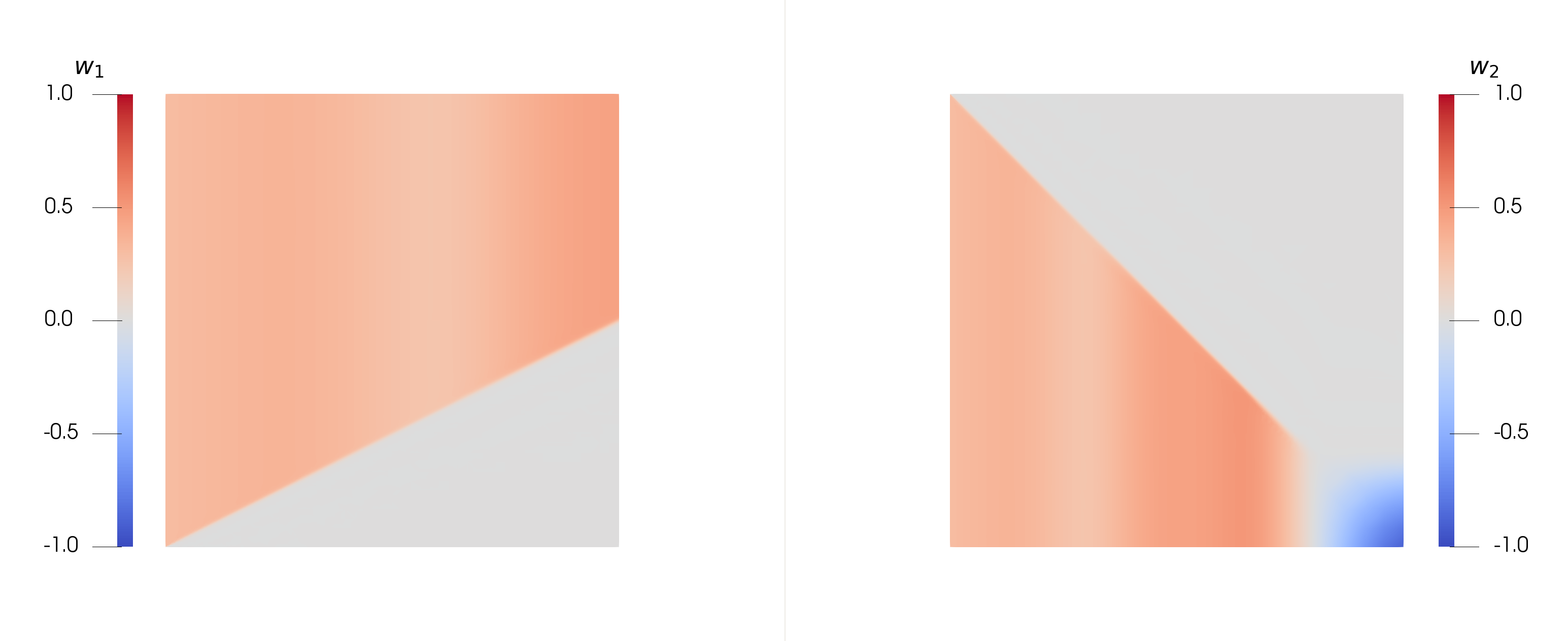}
		\caption*{$t=0.4$}
		\vspace{0.2cm}
		\includegraphics[trim = 0cm 5cm 0cm 0cm, clip, width=0.83\textwidth]{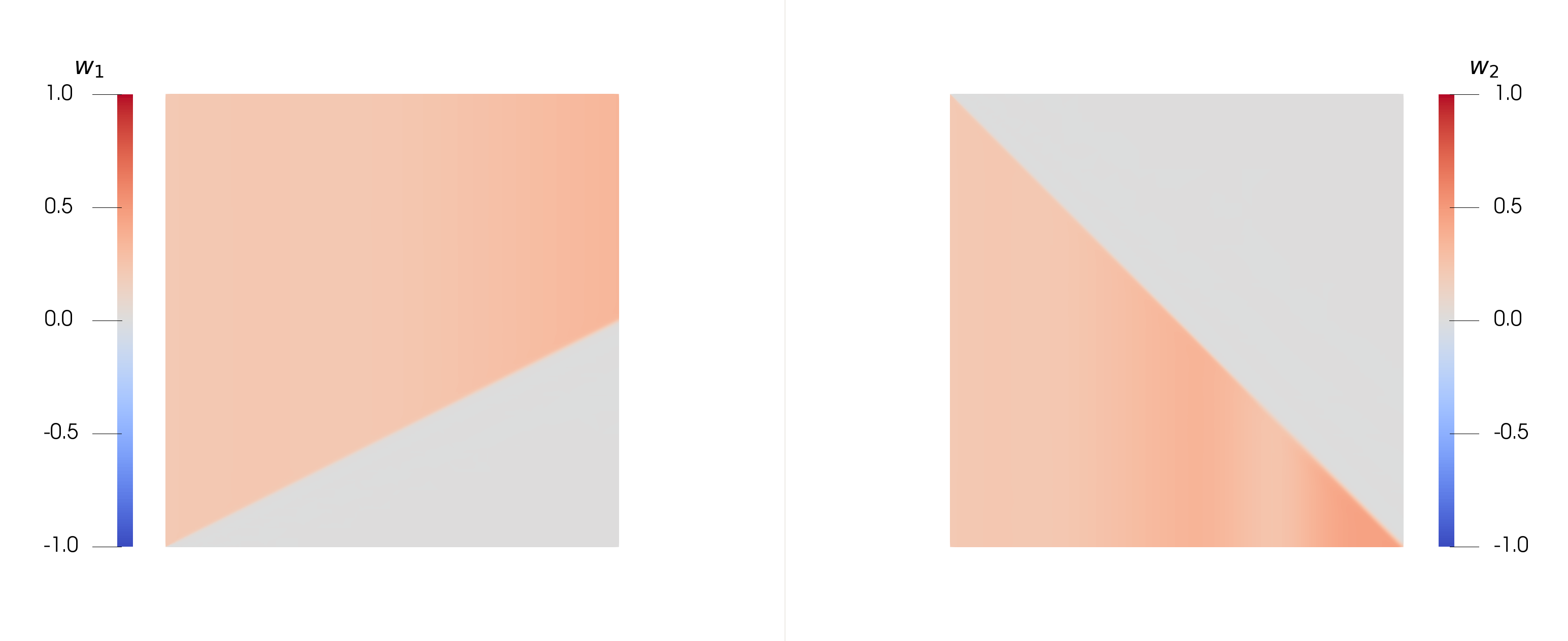}
		\caption*{$t=0.6$}
		\caption{
			Decay of the controlled solution $\V{w}$
	for a simulation with $P=3$, $\operatorname*{CFL} = 0.7$, $\Lmax=6$ and $c_{thresh} = 0.1$ at different times $t$.}
	\label{fig:sol_decayP3}
	\end{figure}
	For the
        problem at hand
	we expect a transport of the first component $w_1$ in the direction $\left(4, 2\right)^T$ and of the second component $w_2$ in $\left(2, -2\right)^T$.
	That is exactly what we see in Figure \ref{fig:sol_decayP3}.
	The boundary values which are given by the control \eqref{integral_for_testcontrol} are transported from the left boundary to the inside of $\Omega$ and ensure that the solution decays.
	At the bottom or at the top we set zero boundary conditions, i.e., the solution equals zero in that region.

\section{Outlook}
\label{sec:concl}

We investigated the influence of numerical dissipation on the decay of the Lyapunov function. For this purpose, we analyzed the stabilization problem using boundary feedback control with viscosity  for the solution of a linear, parabolic multi--dimensional scalar problem. In contrast to the inviscid problem the decay of the Lyapunov function cannot be estimated only by an exponentially decaying term but by an additional integral term accounting for viscosity. \\
	 Since numerical dissipation of any numerical scheme acts similar to physical viscosity, it has been long suspected that numerical dissipation will trigger a similar term in the estimate of the decay of the discrete Lyapunov function. This has now been confirmed and quantified by general  finite volume schemes applied to the inviscid feedback control problem.
Results in the one-dimensional problem have been given as well as for the multi-dimensional case. \\
	 Finally, the analytical results have been confirmed by numerical computations for the one-dimensional and two-dimensional inviscid feedback control problem.
	 The main findings are that numerical dissipation leads to a faster decay of the discrete Lyapunov function in comparison with the decay of the theoretical Lyapunov function.
This becomes most evident for first-order schemes. Increasing the order of the numerical scheme decreases the numerical dissipation and, the decay of the discrete Lyapunov function converges to the decay rate of the  exact Lyapunov function.

\section*{Acknowledgments}
Funded by the Deutsche Forschungsgemeinschaft (DFG, German Research Foundation)
through
SPP 2410 Hyperbolic Balance Laws in Fluid Mechanics: Complexity, Scales, Randomness (CoScaRa) within the Projects 
525842915, 525842915
(Numerical Schemes for Coupled Multi-Scale Problems)
and 525939417 (A Sharp Interface Limit by Vanishing Volume Fraction for 
Non-Equilibrium Two Phase Flows modeled by Hyperbolic Systems of Balance Laws),
 through
GRK2379 (IRTG Hierarchical and Hybrid Approaches in Modern Inverse Problems),
and through
SPP 2183  Eigenschaftsgeregelte Umformprozesse with the Projects  
424334423
Entwicklung eines flexiblen isothermen Reckschmiedeprozesses für die eigenschaftsgeregelte Herstellung von Turbinenschaufeln aus Hochtemperaturwerkstoffen.
\phantomsection
\addcontentsline{toc}{section}{\textbf{References}}
\renewcommand{\refname}{References}
\bibliographystyle{siam}
\bibliography{HHMT-bib,Literature_1}   

\medskip
\medskip

\end{document}